\documentclass[12pt,oneside]{amsart}

\usepackage{amsmath}

\calclayout

\usepackage{amsfonts,amssymb,amsthm,enumitem}
\usepackage{mathtools}
\usepackage{nicefrac,setspace}
\usepackage{xcolor}
\usepackage{tikz}
\usepackage{circuitikz}
\usepackage{amsmath}

\usepackage[draft]{commenting}

\usepackage{bm}

\usepackage{hyperref}       
\usepackage{url}            
\usepackage{booktabs}       
\usepackage{amsfonts}       
\usepackage{nicefrac}       
\usepackage[expansion=false]{microtype}
\usepackage{bookmark}
\numberwithin{equation}{section}
\newtheorem{theorem}{Theorem}[section]

\newtheorem{claim}[theorem]{Claim}
\newtheorem{definition}[theorem]{Definition}

\newtheorem{lemma}[theorem]{Lemma}
\newtheorem{corollary}[theorem]{Corollary}

\newtheorem{prop}[theorem]{Proposition}

\newtheorem{proposition}[theorem]{Proposition}
\newtheorem{remark}[theorem]{Remark}
\newtheorem*{corollary*}{Corollary}
\newtheorem*{theorem*}{Theorem}
\newtheorem*{claim*}{Claim}
\newtheorem*{remark*}{Remark}
\newtheorem*{lemma*}{Lemma}

\newcommand{\showcomments}{yes}

\newsavebox{\commentbox}
{\ifthenelse{\equal{\showcomments}{yes}}%
{\footnotemark
    \begin{lrbox}{\commentbox}
    \begin{minipage}[t]{1.25in}\raggedright\sffamily\tiny
    \footnotemark[\arabic{footnote}]}
{\begin{lrbox}{\commentbox}}}%
{\ifthenelse{\equal{\showcomments}{yes}}%
{\end{minipage}\end{lrbox}\marginpar{\usebox{\commentbox}}}
{\end{lrbox}}}

\declareauthor{ms}{Mike}{blue}
\declareauthor{gy}{Gal}{red}


\date{\today}

\title{Growth of Approximate Groups in Hyperbolic Groups}
\author{Michael Saks}
\address{Dept of Math, Rutgers University, New Brunswick, USA}
\email{saks@math.rutgers.edu}

\author{Gal Yehuda}
\address{Dept of Math, Yale, New Haven, USA}
\email{gal.yehuda@yale.edu}

\begin{document}
\vspace*{-1cm} 

\begin{abstract}
We prove a growth dichotomy for infinite approximate groups, and more generally approximate semigroups, in hyperbolic groups. If \(G\) is a finitely generated hyperbolic group and \(A\subseteq G\) is infinite with
\[
        A^2\subseteq AX
\]
for some finite \(X\subseteq G\), then either \(\langle A\rangle\) is virtually cyclic, or \(A\) has positive exponential growth in the ambient word metric.

We also introduce a product-growth criterion for the existence of growth rates of approximate semigroups. The criterion applies to hyperbolic groups: if \(G\) is hyperbolic with finite generating set \(S\), then there is a constant \(c_{G,S}>0\) such that
\[
        |UV| \geq c_{G,S}\,\frac{|U||V|}{n+k+1},
        \qquad U\subseteq B_n,\; V\subseteq B_k.
\]
The linear loss is optimal in order whenever \(G\) contains an element of infinite order. In the free group with its standard generating set one may take \(c_{G,S}=1/4\). We also prove that, in a free group, if \(U\subseteq S_n\) and \(V\subseteq S_k\), then
\[
        |UV|\geq
        \left(\frac{2}{3}+\frac{1}{3\cdot 4^{\min\{n,k\}}}\right)|U||V|,
\]
and this constant is sharp for all \(n,k\).
\end{abstract}
\maketitle

\vspace{-.4cm}

\section{Introduction}
Approximate groups are one of the basic objects of modern additive
combinatorics. In the finite theory, small doubling forces strong algebraic
structure: foundational work of Hrushovski and of Breuillard--Green--Tao shows
that finite approximate groups are controlled by nilpotent models \cite{hrushovski2012stable,breuillard2012structure}. 
Much less is known about the structure of approximate groups in infinite groups. In particular, it is natural to ask what the small-doubling
relation alone implies about the large-scale geometry and asymptotic growth of
an infinite subset of a finitely generated group.

This paper studies that question in the negatively curved setting. Our main
result shows that, inside a hyperbolic group, an infinite set satisfying the
one-sided small-doubling relation
\[
        A^2\subseteq AX
\]
already obeys a sharp qualitative alternative: either it generates a virtually
cyclic subgroup, or it has positive exponential growth in the ambient word
metric. No geometric finite generation, hull construction, uniform
discreteness, approximate-lattice hypothesis, or a priori action is assumed.
The conclusion is extracted directly from the internal product-set relation
\(A^2\subseteq AX\) in the Cayley graph.

A second theme of the paper is additive-combinatorial. We study how small an
ordinary product set \(UV\) can be when \(U\) and \(V\) are arbitrary finite
subsets of balls in a finitely generated group. This leads to a product-growth
criterion which implies the existence of growth rates for approximate
semigroups, and to sharp product-set estimates in free and hyperbolic groups.
Thus the paper connects three phenomena: small doubling of infinite sets,
product growth of finite subsets of balls, and exponential growth of ambient
counting functions.

More precisely, let \(G\) be a finitely generated group with a fixed finite generating set, and let
\(B_n\) denote the ball of radius \(n\) around the identity in the corresponding Cayley
graph.  For a subset \(A\subseteq G\), we study the counting function
\[
        n\longmapsto |A\cap B_n|.
\]
Some basic questions of interest are:
\begin{itemize}
    \item Does \(A\) have \emph{positive exponential growth in \(G\)}, i.e. do there exist constants
\(C>0\) and \(\lambda>1\) such that
\[
        |A\cap B_n|\ge C\lambda^n,
\]
or  equivalently, 
\[
        \liminf_{n\to\infty}\frac{1}{n}\log |A\cap B_n|>0.
\]
\item Does $A$ have a \emph{well-defined exponential growth rate}, i.e., does
\[\lim_{n\to\infty}\frac{1}{n}\log|A \cap B_n|\] 
exist? 
(Note that this limit may be 0, in which case $A$ has a well-defined exponential growth rate but does not have positive exponential growth.)
\end{itemize}

When \(A\) is a subgroup, these questions are instances of the classical theory of
relative subgroup growth.  If \(H\leq G\), one studies
\[
        \gamma_H^G(n)=|H\cap B_n|.
\]
For subgroups of hyperbolic groups, the existence of the relative exponential growth
rate is known in considerable generality.  Schesler proved that
\[
        \lim_{n\to\infty}\sqrt[n]{\gamma_H^G(n)}
\]
exists for every subgroup \(H\) of an acylindrically hyperbolic group \(G\) which
contains a loxodromic element \cite{schesler2022relative}.  Together with the
standard subgroup dichotomy in hyperbolic groups, this implies that subgroups of
hyperbolic groups have well-defined relative exponential growth rates; the rate is
greater than \(1\) precisely in the non-elementary, equivalently non-virtually-cyclic,
case \cite{kapovich2002boundaries}.  For free groups, the possible values are now
understood much more finely: it was proven in \cite{louvaris2024density,louvaris2025every} that the set of growth
rates of finitely generated subgroups of \(F_r\) is dense in \([1,2r-1]\), and later
that every value in this interval is realized by some subgroup, in general infinitely
generated.

In this paper we investigate the above questions for approximate groups, and more generally, approximate semigroups.  
A subset \(A\) of a group \(G\) is an \emph{approximate semigroup} if there is a finite subset \(X\subseteq G\) such that
\[
        A^2\subseteq AX,
\]
and it is an \emph{approximate group} if, in addition, \(1\in A\) and \(A=A^{-1}\).

We use the right-handed convention \(A^2\subseteq AX\).  For approximate groups in the usual sense this is equivalent to the left-handed convention \(A^2\subseteq XA\): taking inverses changes \(A^2\subseteq AX\) into \(A^2\subseteq X^{-1}A\).  For approximate semigroups the two conventions are not literally equivalent in a fixed group, but the results and proofs in this paper are insensitive to this choice; the left-handed version is obtained by passing to the opposite group.  We use the right-handed convention because it has the following geometric interpretation.  If \(K=\max_{x\in X}|x|\) and \(ab=a'x\), with \(a'\in A\) and \(x\in X\), then
\[
        d(ab,a')=|x|\le K.
\]
Thus every product of two elements of \(A\) lies a uniformly bounded distance from \(A\) in the Cayley graph.

For approximate groups and approximate semigroups, answers to the above basic asymptotic questions are not automatic.  The relation \(A^2\subseteq AX\) does not make \(A\)
a subgroup, and the counting function \(n\mapsto |A\cap B_n|\) need not satisfy the usual submultiplicative estimates available for finitely generated groups.  Thus one is led to ask for which ambient groups \(G\) approximate groups \(A\subseteq G\) have a well-defined exponential growth rate, and, in the hyperbolic setting, whether positive exponential growth is characterized by a natural algebraic condition on
\(\langle A\rangle\).  A point of the results below is that these conclusions are obtained directly from the internal small-doubling relation \(A^2\subseteq AX\) in
the ambient Cayley graph.

Our first main result is a growth dichotomy in hyperbolic groups.

\begin{theorem}[Hyperbolic growth dichotomy]
\label{thm:hyp dichotomy}
Let \(G\) be a finitely generated hyperbolic group, and let \(A\subseteq G\) be an approximate semigroup. Then either
\begin{enumerate}[label=\textup{(\arabic*)}]
    \item \(A\) has positive exponential growth in \(G\), or
    \item \(\langle A\rangle\) is virtually cyclic.
\end{enumerate}
\end{theorem}
As far as we know, this is the first dichotomy of this form for arbitrary infinite
subsets of hyperbolic groups under only a one-sided small-doubling hypothesis.  There is
now a substantial literature on infinite approximate subgroups, but much of it develops
in a different direction.  One important strand concerns approximate lattices and
discrete approximate subgroups in locally compact groups, beginning with work of
Björklund and Hartnick and continuing through structural results in Lie and linear
settings.  In particular, Machado developed a broad Lie-theoretic approach to infinite
approximate subgroups and approximate lattices: his work on good models, nilpotent Lie
groups, soluble Lie groups, higher-rank settings, and linear groups gives powerful
structural tools in those contexts
\cite{bjorklund2018approximate,machado2020good,machado2020nilpotent,
machado2022soluble,machado2023higherrank,machado2023linear}.  Another important strand
is the geometric theory of infinite approximate groups developed by Cordes, Hartnick,
Toni\'c and Machado, based on geometric finite generation and quasi-isometric
quasi-actions; in that framework one obtains, among other things, exponential growth
results for non-elementary hyperbolic approximate groups
\cite{cordes2024foundations}.  The present paper is complementary to these works: we
work inside a fixed finitely generated ambient group and use only the internal relation
\(A^2\subseteq AX\), with no geometric finite generation, no hull construction, no
uniform discreteness, and no a priori action attached to \(A\).

The proof of the hyperbolic dichotomy is based on a quantitative Schottky principle
adapted to approximate groups.  In a free-semigroup argument it is enough to produce
many distinct positive words.  Here that is not sufficient.  The relation
\(A^2\subseteq AX\) allows long products of elements of \(A\) to be projected back into
\(A\), but each projection carries an error from \(X\), and this error grows linearly
with the word length.  The main point is therefore quantitative separation: one must
construct words whose values in the Cayley graph remain far enough apart to survive this
linear projection error.  Hyperbolicity supplies such separation once one finds an
element with sufficiently large stable translation length.  The opposite case is ruled
out by a boundary argument showing that, unless stable translation lengths become
unbounded in \(A\) or \(A^2\), all of \(A\) fixes a common boundary point, forcing
\(\langle A\rangle\) to be elementary.

Our second theme is the existence of growth rates for approximate groups in more general
ambient groups.  We isolate a product-set condition which captures the required
almost-superadditivity.

\begin{definition}
A finitely generated group \(G\) has \emph{large product growth} if there exists a
polynomial \(P\) such that, for all finite subsets \(A\subseteq B_n\) and
\(B\subseteq B_k\),
\[
        |AB|\ge \frac{|A|\,|B|}{P(n+k)}.
\]
\end{definition}

This condition says that multiplication of finite sets cannot collapse too much, except
for a polynomial loss in the ambient radii.  

This condition is closely related to Sapir's Rapid Expansion property
\cite[Section~4]{sapir2015rapid}.  In the finitely generated setting, Rapid
Expansion asks for a polynomial \(P\) such that, whenever \(S\subseteq B_r\)
and \(X\subseteq G\) are finite,
\[
        |SX|\ge \frac{|S|\,|X|}{P(r)}.
\]
Thus Rapid Expansion is a one-sided version in which the bound is independent
of the radius of the second factor.  In particular, Rapid Expansion implies
large product growth.

Large product growth is additionally related in spirit to Schesler's ambiguity function, but it is
different in both formulation and use \cite{schesler2022relative}.  Schesler studies
relative growth of subgroups \(H\leq G\), where the counting sets are relative balls
\(H\cap B_n\).  Thus the objects being multiplied are not arbitrary finite subsets of the ambient group, but initial segments of a fixed subgroup in the ambient word
metric.  Moreover, since the ordinary multiplication map
\[
        (H\cap B_m)\times(H\cap B_n)\longrightarrow H,\qquad (h,k)\mapsto hk,
\]
may have very large fibers, Schesler replaces it by a concatenation map
\[
        \Phi(h,k)=h x_{h,k} k,
\]
where the connecting element \(x_{h,k}\) has uniformly bounded word length, and then bounds the fibers of \(\Phi\).  This gives weak supermultiplicativity of relative balls, because
\[
        \Phi((H\cap B_m)\times(H\cap B_n))\subseteq H\cap B_{m+n+O(1)}.
\]
Equivalently, after restricting the same map to arbitrary finite
\(U,V\subseteq H\), the ambiguity method gives a lower bound for a boundedly thickened product \(U B_{O(1)}V\), not for the ordinary product \(UV\).

By contrast, large product growth is an ambient product-set condition: it requires a
uniform lower bound for the ordinary product \(UV\) for all finite \(U\subseteq B_m\) and \(V\subseteq B_n\), with no assumption that \(U\) and \(V\)
lie in a common subgroup or have the form \(H\cap B_m\), \(H\cap B_n\).  This distinction is essential for approximate groups.  If \(A\) is an approximate group, then \(A\cap B_m\) and \(A\cap B_n\) are arbitrary finite pieces of \(A\), not balls
in a subgroup, and the relation \(A^2\subseteq AX\) controls ordinary products
\[
        (A\cap B_m)(A\cap B_n),
\]
up to multiplication by the fixed error set \(X\).  It gives no comparable control on products with an inserted connector,
\[
        (A\cap B_m)E(A\cap B_n),
\]
even when \(E\) is finite and bounded.  Thus both approaches produce almost supermultiplicative inequalities, but by different mechanisms: the ambiguity
method bypasses cancellation by inserting a bounded connector in the subgroup setting, whereas our product-growth condition requires a direct lower bound for
ordinary products \(UV\) of arbitrary finite subsets.

Our abstract result is the following.

\begin{theorem}[Existence of growth rates]
Let \(G\) be a finitely generated group with large product growth, and let
\(A\subseteq G\) be a non-empty approximate semigroup. Then the limit
\[
        \omega(A):=\lim_{n\to\infty}\frac{1}{n}\log |A\cap B_n|
\]
exists.
\end{theorem}

In Section~\ref{sec:growth-rates} we prove this
theorem in a more general form: it is enough to have product
``shrinkage'' at most a non-decreasing function \(f\) such that
\[
        \sum_{n=1}^{\infty}\frac{\log f(n)}{n^2}<\infty.
\]

No symmetry is used in the proof of this theorem. The approximate-semigroup
hypothesis \(A^2\subseteq AX\) is enough: it turns products
\((A\cap B_n)(A\cap B_k)\) into points of \(A\), up to multiplication by the
fixed finite correction set \(X\). 
The large product growth condition prevents too many such products from collapsing. Together these two facts give the
almost-superadditivity needed to prove the existence of the limit.

Large product growth for a group $G$ has the following further consequence for its
approximate groups and semigroups: if an approximate semigroup $A$ in $G$ has
exponential growth rate zero, then \(|A\cap B_n|\) is polynomially bounded.  Thus within $G$, there are noapproximate semigroups with intermediate subexponential growth, even though such growth can occur for arbitrary finitely generated groups or arbitrary subsets.

One broad source of large product growth comes from the Rapid Decay (RD) property.  Recall
that a finitely generated group \(G\) has property \(\mathrm{RD}\), with respect to the
word metric, if convolution by functions supported in a ball of radius \(R\) has
operator norm bounded by a polynomial in \(R\) times the \(\ell^2\)-norm.  Equivalently,
one has polynomial convolution estimates for functions supported in balls.  Applying
these estimates to indicator functions, and then using the elementary
Cauchy--Schwarz energy inequality, gives a product-set bound of the form
\[
        |AB|\ge \frac{|A|\,|B|}{P(\min\{n,k\})}
\]
whenever \(A\subseteq B_n\) and \(B\subseteq B_k\).  Thus property
\(\mathrm{RD}\) implies large product growth; this implication was
previously isolated by Sapir as the Rapid Expansion property
\cite[Section~4]{sapir2015rapid}.
This gives a flexible analytic route to
the existence theorem above, and applies in particular to hyperbolic groups, free
groups, groups of polynomial growth, and many other groups arising in geometric group
theory
\cite{Jolissaint1990,Nica2010,chatterji2017introRD}.

The product-growth estimates used here are also product-set theorems in their
own right. They ask a basic non-commutative additive-combinatorial question:
given finite sets \(U\subseteq B_n\) and \(V\subseteq B_k\), how small can the
ordinary product set \(UV\) be? Equivalently, how much can the multiplication
map
\[
        U\times V \longrightarrow UV
\]
collapse because of cancellation?

Rapid Decay gives a robust polynomial-loss answer in great generality, but it
does not detect the optimal radius dependence in negatively curved groups. We
therefore prove direct product-set estimates in free and hyperbolic groups.
For free groups these estimates are essentially sharp in balls and completely
sharp on spheres.

\begin{theorem}[Product growth in free groups]\label{thm:free-product-growth}
Let \(A,B\subseteq F_r\) be finite non-empty sets with
\[
        A\subseteq B_n,\qquad B\subseteq B_k.
\]
Then
\[
        |AB|\ge \frac{|A|\,|B|}{4(n+k)+1}.
\]
Moreover, if \(A\subseteq S_n\) and \(B\subseteq S_k\), and if
\(m=\min\{n,k\}\), then
\[
        |AB|\ge \left(\frac23+\frac{1}{3\cdot 4^m}\right)|A|\,|B|.
\]
The constant in the sphere-sphere estimate is sharp for every \(m\).
\end{theorem}

The first inequality is optimal in its radius dependence: intervals inside a
cyclic subgroup show that no uniform sublinear loss in \(n+k\) is possible.
Thus the ball--ball estimate gives the correct order of magnitude for the
worst possible collapse of \(U\times V\to UV\).

The sphere--sphere estimate is stronger. In that case there is no growing
radius loss at all: the optimal shrinkage is a constant depending only on
\(m=\min\{n,k\}\), and the theorem determines that constant exactly for every
\(m\). The lower bound is therefore not merely sharp up to constants; it
completely solves the sphere--sphere product-growth problem in free groups.

The same linear-order ball-ball estimate holds, with a group-dependent constant, in
all hyperbolic groups.

\begin{theorem}[Product growth in hyperbolic groups]
\label{thm:hyperbolic-product-growth}
Let \(G\) be a finitely generated hyperbolic group, with a fixed finite generating set $S$.
Then there exists a constant \(c_{G,S}>0\) such that, for all finite non-empty sets
\(U\subseteq B_n\) and \(V\subseteq B_k\),
\[
        |UV|\ge c_{G,S}\frac{|U|\,|V|}{n+k+1}.
\]
In particular, every hyperbolic group has large product growth.
\end{theorem}

Since free groups are hyperbolic, Theorem~\ref{thm:hyperbolic-product-growth}
implies Theorem~\ref{thm:free-product-growth}, up to the value of the constant.  We nevertheless keep the free-group theorem as a separate result.
First, the free-group estimates are sharper: in balls we obtain an explicit absolute linear bound, and on spheres we determine the exact best constant.  Second, the proof in free groups is the clearest form of the argument.  The hyperbolic proof follows the
same strategy, but is technically more involved.

Nica recently proved a full-sphere expansion theorem for hyperbolic groups, 
This theorem shows large product growth when the left factor is an entire  sphere, and
the right factor is an arbitrary finite subset of the group:

\begin{theorem}[Expansion of full spheres in hyperbolic groups]\cite[Theorem~2.5]{Nica2024}
\label{thm:intro-hyperbolic-sphere-expansion}
Let \(G\) be a non-elementary hyperbolic group, with a fixed finite generating
set \(S\), and write
\[
        S_n=\{g\in G:|g|=n\}.
\]
Then there exists a constant \(c'_{G,S}>0\) such that, for every \(n\ge 0\) and
every finite non-empty \(X\subseteq G\),
\[
        |S_nX|\ge c'_{G,S}|S_n|\,|X|.
\]
\end{theorem}

Nica's proof
uses analytic estimates for spherical averaging operators, and he asks in
\cite[Remark~8.4]{Nica2024} whether such expansion bounds can be proved
directly, without functional-analytic detours.  Our proof uses the same
bounded-cancellation and bounded-fiber mechanism as
Theorem~\ref{thm:hyperbolic-product-growth}.  The additional input is that full
spheres have small shadows: after fixing a sufficiently large cancellation
threshold, a positive proportion of \(S_n\) has uniformly bounded cancellation
with each fixed element of \(X\).

These direct estimates are substantially sharper than what follows from the general Rapid Decay argument. 
The full-sphere expansion theorem above is another instance of this geometric
approach.
Property RD gives a flexible polynomial-loss
lower bound for product sets, uniform across all groups with RD. In
hyperbolic groups, the geometry of cancellation gives more: for subsets of
balls one obtains the optimal linear-order loss, while for spherical layers
one obtains a radius-independent lower bound. Thus the free and hyperbolic
product estimates are not simply applications of analytic RD machinery; they
are sharper additive-combinatorial product theorems for negatively curved
groups.

The Rapid Decay implication used above is due to Sapir.  In the final part of the paper we recall the implication in our notation and record the following explicit large-product-growth consequence.
\begin{theorem}[Rapid Decay implies large product growth]
\label{thrm:Rapid-Decay}
Let \(G\) be a finitely generated group with property \(\mathrm{RD}^s\), with constant
\(C>0\).  If \(A\subseteq B_n\) and \(B\subseteq B_k\) are finite non-empty subsets of
\(G\), then
\[
        |AB|\ge \frac{|A|\,|B|}{C^2(1+
        \min\{n,k\})^{2s}}.
\]
In particular, \(G\) has large product growth.
\end{theorem}

Consequently, every approximate group in a finitely generated group with property
\(\mathrm{RD}\) has a well-defined exponential growth rate.

Taken together, the results give an asymptotic theory of infinite approximate
semigroups in negatively curved and product-expanding groups. 
In hyperbolic groups, one-sided small doubling alone forces a dichotomy between virtual cyclicity and exponential growth. In general finitely generated groups, the existence of growth rates follows from quantitative lower bounds for ordinary
product sets. The free and hyperbolic estimates show that these product-set
bounds can be sharp additive-combinatorial statements.

\section{Preliminaries for the hyperbolic dichotomy}
\label{sec:prelims}
In this section we fix notation and recall the elementary facts from hyperbolic geometry and approximate semigroups that will be used in the proof of Theorem~\ref{thm:hyp dichotomy}.

\subsection{Notation for products}
\label{subsec:general}
Throughout, \(G\) denotes a finitely generated group with identity \(1\), and
\(S\) denotes a fixed finite symmetric generating set.  The word metric is always
computed with respect to this generating set.

For subsets \(U,V\subseteq G\), we write
\[
        UV=\{uv:u\in U,\ v\in V\}.
\]
For \(j\geq 0\), let \(U^{\times j}\) denote the set of \(j\)-tuples with entries in
\(U\), and set
\[
        U^*=\bigcup_{j\geq 0}U^{\times j}.
\]
The unique element of \(U^{\times 0}\) is the empty tuple.

We use boldface letters for tuples.  If
\[
        \mathbf u=(u_1,\ldots,u_j)\in G^*,
\]
we write
\[
        \ell(\mathbf u)=j
\]
for its length as a tuple, and
\[
        \overline{\mathbf u}=u_1\cdots u_j\in G
\]
for the group element represented by the tuple.  For the empty tuple we set
\(\overline{\mathbf u}=1\).  If
\(\mathbf u=(u_1,\ldots,u_j)\) and
\(\mathbf v=(v_1,\ldots,v_k)\), their concatenation is
\[
        \mathbf u\circ \mathbf v=(u_1,\ldots,u_j,v_1,\ldots,v_k).
\]
For \(m\geq 1\), let \(\mathbf u^{\circ m}\) denote the concatenation of \(m\)
copies of \(\mathbf u\).  We also define
\[
        \mathbf u^{-1}=(u_j^{-1},\ldots,u_1^{-1}).
\]
Then
\[
        \overline{\mathbf u^{-1}}=\overline{\mathbf u}^{-1},
        \qquad
        \overline{\mathbf u\circ\mathbf v}
        =
        \overline{\mathbf u}\,\overline{\mathbf v}.
\]
When no confusion can arise, we write
\[
        |\mathbf u|:=|\overline{\mathbf u}|,
        \qquad
        d(\mathbf u,\mathbf v):=d(\overline{\mathbf u},\overline{\mathbf v}).
\]
Thus
\[
        d(\mathbf u,\mathbf v)
        =
        |\overline{\mathbf u}^{-1}\overline{\mathbf v}|
        =
        |\mathbf u^{-1}\circ\mathbf v|.
\]
This convention will only be used for tuples.

\subsection{The Cayley graph metric and Gromov products}
\label{subsec:Cayley}
Let \(\Gamma=\Gamma(G,S)\) be the Cayley graph of \(G\) with respect to \(S\), and
let \(d\) be its graph metric.  We write
\[
        |g|=d(1,g),
        \qquad
        B_r=\{g\in G: |g|\leq r\}.
\]
For \(x,y,t\in G\), the Gromov product of \(x\) and \(y\) based at \(t\) is
\[
        (x\mid y)_t
        =
        \frac12\bigl(d(x,t)+d(y,t)-d(x,y)\bigr).
\]
When \(t=1\), we omit the subscript and write \((x\mid y)\).

We record some elementary identities and inequalties.

\begin{prop}
\label{prop:easy metric}
For all \(x,y,t\in G\), the following hold.
\begin{enumerate}[label=\textup{(\arabic*)}]
    \item \label{easy metric 5} \(d(x,y)=|x^{-1}y|\).
    \item \label{easy metric 1}
    \[
        (x\mid y)_t
        =
        \frac12\bigl(|t^{-1}x|+|t^{-1}y|-|x^{-1}y|\bigr).
    \]
    \item \label{easy metric 2}
    \[
        |xy|=|x|+|y|-2(x^{-1}\mid y).
    \]
    \item \label{easy metric 3}
    \[
        (x\mid y)_t\leq \min\{d(x,t),d(y,t)\}.
    \]
    \item \label{easy metric 4}
    \[
        (xy\mid x)=|x|-(x^{-1}\mid y)\geq |x|-|y|.
    \]
\end{enumerate}
\end{prop}

\begin{proof}
All statements follow directly from the definition of the Gromov product and from
left-invariance of the word metric.
\end{proof}

For \(g\in G\), the \emph{stable translation length} of \(g\) is
\[
        \tau(g)=\lim_{n\to\infty}\frac{|g^n|}{n}.
\]
The limit exists because the sequence \(a_n=|g^n|\) is subadditive:
\[
        |g^{n+m}|\leq |g^n|+|g^m|.
\]
By Fekete's lemma,
\[
        \tau(g)=\inf_{n\geq 1}\frac{|g^n|}{n}.
\]
We shall use the following consequence.

\begin{prop}
\label{prop:tau}
For every \(g\in G\) and every \(n\geq 1\),
\[
        |g^n|\geq n\tau(g).
\]
\end{prop}

\subsection{Hyperbolic groups and boundary facts}
\label{subsec:hyp}
We say that \(G\) is \(\delta\)-hyperbolic, with respect to the generating set \(S\),
if
\begin{equation}
\label{eqn:hyperbolic}
        (x\mid z)_t
        \geq
        \min\{(x\mid y)_t,(y\mid z)_t\}-\delta
\end{equation}
for all \(x,y,z,t\in G\).  Hyperbolicity is independent of the finite generating set,
although the constant \(\delta\) depends on the generating set.  Throughout this
subsection and in the proof of Theorem~\ref{thm:hyp dichotomy}, we assume that \(G\) is
\(\delta\)-hyperbolic.

We will use the following immediate consequences of hyperbolicity.

\begin{prop}
\label{prop:easy hyp}
For all \(w,x,y,z\in G\), the following hold.
\begin{enumerate}[label=\textup{(\arabic*)}]
    \item \label{easy hyp 1}
    \[
        (w\mid z)
        \geq
        \min\{(w\mid x),(x\mid y),(y\mid z)\}-2\delta.
    \]
    \item \label{easy hyp 2}
    If
    \[
        (y\mid z)<(x\mid z)-\delta,
    \]
    then
    \[
        (x\mid y)\leq (y\mid z)+\delta.
    \]
\end{enumerate}
\end{prop}

\begin{proof}
Part \textup{(1)} is obtained by applying \(\delta\)-hyperbolicity twice.  For
\textup{(2)}, apply \eqref{eqn:hyperbolic} to the triple \(y,x,z\):
\[
        (y\mid z)\geq \min\{(y\mid x),(x\mid z)\}-\delta.
\]
If \((x\mid y)>(y\mid z)+\delta\), then, together with the hypothesis
\((x\mid z)>(y\mid z)+\delta\), both terms inside the minimum are larger than
\((y\mid z)+\delta\).  This would force \((y\mid z)>(y\mid z)\), a
contradiction.
\end{proof}

The next estimate is the form of the local-to-global principle needed in the proof.

\begin{prop}[Piecewise-geodesic lower bound]
\label{prop:hyp-piecewise}
Let
\[
        \mathbf s=(s_1,\ldots,s_n)\in G^{\times n}.
\]
Suppose that
\[
        |s_i|>2M+2\delta
        \qquad \text{for }1\leq i\leq n,
\]
and
\[
        (s_i^{-1}\mid s_{i+1})\leq M
        \qquad \text{for }1\leq i<n.
\]
Then
\[
        |\mathbf s|
        \geq
        \sum_{i=1}^n |s_i|-2(M+\delta)(n-1).
\]
\end{prop}

\begin{proof}
Set \(p_r=s_1\cdots s_r\).  We first prove that, for \(1\leq r<n\),
\[
        (p_r^{-1}\mid s_{r+1})\leq M+\delta.
\]
For \(r=1\), this is one of the hypotheses.  Suppose \(r>1\) and that
\[
        (p_{r-1}^{-1}\mid s_r)\leq M+\delta.
\]
Using Proposition~\ref{prop:easy metric}\textup{(\ref{easy metric 2})}, we have
\[
        |p_r|=|p_{r-1}s_r|
        \geq
        |p_{r-1}|+|s_r|-2(M+\delta).
\]
Therefore
\[
\begin{aligned}
        (p_r^{-1}\mid s_r^{-1})
        &=
        \frac12\bigl(|p_r|+|s_r|-|p_{r-1}|\bigr)  \\
        &\geq
        |s_r|-(M+\delta)
        >
        M+\delta.
\end{aligned}
\]
Since \((s_r^{-1}\mid s_{r+1})\leq M\), Proposition~\ref{prop:easy hyp}\textup{(\ref{easy hyp 2})}, applied to
\[
        x=p_r^{-1},
        \qquad
        y=s_{r+1},
        \qquad
        z=s_r^{-1},
\]
gives
\[
        (p_r^{-1}\mid s_{r+1})\leq M+\delta.
\]
This proves the claim.

We now prove the proposition by induction on \(n\).  The case \(n=1\) is immediate.
For \(n>1\), using the claim with \(r=n-1\) and applying the induction hypothesis
to \((s_1,\ldots,s_{n-1})\), we obtain
\[
\begin{aligned}
        |\mathbf s|
        =|p_n|
        &=|p_{n-1}s_n| \\
        &=|p_{n-1}|+|s_n|-2(p_{n-1}^{-1}\mid s_n) \\
        &\geq
        \sum_{i=1}^n |s_i|-2(M+\delta)(n-1).
\end{aligned}
\]
\end{proof}

We next recall the boundary language used below.  If
\(x=(x_i)_{i\geq 1}\) and \(y=(y_i)_{i\geq 1}\) are sequences in \(G\), 
\[
        (x\mid y)_t
        =
        \liminf_{i,j\to\infty}(x_i\mid y_j)_t.
\]
A sequence \(x=(x_i)\) is a \emph{Gromov sequence} if
\[
        (x\mid x)=\infty.
\]
This condition is independent of the basepoint.  Two Gromov sequences \(x,y\) are
\emph{asymptotic} if
\[
        (x\mid y)_t=\infty,
\]
which also does not depend on the basepoint.
Hyperbolicity of $G$ implies that the asymptotic relation is an equivalence relation on the set of Gromov sequences.
The Gromov boundary \(\partial G\) is the set of asymptotic classes of Gromov
sequences.  If \(x=(x_i)\) represents \(\xi\in\partial G\), we write
\[
        x_i\to \xi.
\]

\begin{prop}
\label{prop:subsequence}
Every sequence \((x_i)\) in \(G\) with \(|x_i|\to\infty\) along a subsequence has a
Gromov subsequence.
\end{prop}

\begin{lemma}[Boundary separation]
\label{cor:hyp-boundary-separation}
If \(x\) and \(y\) are Gromov sequences that are not asymptotic, then there exists
\(M\geq 0\) such that
\[
        (x_i\mid y_j)\leq M
        \qquad \text{for all }i,j\geq 1.
\]
\end{lemma}
This is a standard fact, see for example \cite{vaisala2005gromov}.

In a hyperbolic group, elements with positive stable translation length are called
\emph{loxodromic}.  We shall use the standard fact that, for \(g\in G\),
\[
        \tau(g)>0
        \qquad\Longleftrightarrow\qquad
        g\text{ has infinite order}.
\]

\begin{prop}
\label{prop:loxodromic}
In a hyperbolic group, \(\tau(g)>0\) if and only if \(g\) has infinite order.
\end{prop}

For a loxodromic element \(g\), we write \(g^+,g^-\in\partial G\) for the limits
\[
        g^n\to g^+,
        \qquad
        g^{-n}\to g^-.
\]
Left multiplication induces the usual action of \(G\) on \(\partial G\).  Right
multiplication does not give a new action on the boundary; rather, it moves every
sequence only a bounded distance.

\begin{proposition}[Right multiplication and boundary limits]
\label{prop:hyp-right-mult-limit}
Let \((x_i)\) be a Gromov sequence and let \(h\in G\).  Then \((x_i h)\) is a
Gromov sequence asymptotic to \((x_i)\).  In particular, if \(x_i\to \xi\in\partial G\),
then
\[
        x_i h\to \xi.
\]
\end{proposition}

\begin{proof}
For every \(i\),
\[
        d(x_i,x_i h)=|h|.
\]
Thus the two sequences remain a uniformly bounded distance apart.  Hence they are
asymptotic and define the same boundary point.
\end{proof}

We also use the following standard boundary facts.  Recall that a subgroup of a
hyperbolic group is called \emph{elementary} if it is finite or virtually cyclic.

\begin{lemma}[Boundary dynamics of loxodromic elements]
\label{lem:lox}
Let \(G\) be a hyperbolic group.
\begin{enumerate}[label=\textup{(\arabic*)}]
    \item \label{two fixed points}
    A loxodromic element \(g\) has exactly two fixed points on \(\partial G\), namely
    \(g^+\) and \(g^-\).
    \item \label{shared-fixed-point}
    If \(g,h\in G\) are loxodromic, then either
    \[
        \{g^+,g^-\}=\{h^+,h^-\},
    \]
    or
    \[
        \{g^+,g^-\}\cap\{h^+,h^-\}=\emptyset.
    \]
\end{enumerate}
\end{lemma}

\begin{lemma}[Elementary subgroups]
\label{lem:hyp-elementary-facts}
Let \(G\) be a hyperbolic group.
\begin{enumerate}[label=\textup{(\arabic*)}]
    \item For every \(\xi\in\partial G\), the stabilizer
    \[
        \operatorname{Stab}_G(\xi)
    \]
    is elementary.
    \item If \(H\leq G\) preserves a two-point subset of \(\partial G\), then \(H\) is
    elementary.
\end{enumerate}
\end{lemma}

These are standard consequences of the classification of elementary subgroups of
hyperbolic groups; see, for instance, \cite[Section~3]{kapovich2002boundaries}.

\subsection{Approximate semigroups and approximate products}
\label{subsec:approx groups}
A subset \(A\subseteq G\) is an \emph{approximate semigroup} if there is a finite set
\(X\subseteq G\) such that
\[
        A^2\subseteq AX.
\]
We call \(X\) a \emph{correction set} for \(A\).  If, in addition, \(1\in A\) and
\(A=A^{-1}\), then \(A\) is an approximate group.

\begin{remark}
\label{rem:right-left-convention}
We use the right-handed convention \(A^2\subseteq AX\).  For approximate groups
this is equivalent to the left-handed convention, after replacing \(X\) by \(X^{-1}\).
For approximate semigroups the two conventions differ in a fixed group, but passing
to the opposite group exchanges them.  Since the word metric is unchanged under this
identification, the results below have identical left-handed formulations.
\end{remark}

For the rest of this subsection, let \(A\) be an approximate semigroup with correction
set \(X\), and set
\[
        K=\max_{x\in X}|x|.
\]
Enlarging \(X\), if necessary, we may assume that \(1\in X\).  We may also assume
that \(1\in A\): replacing \(A\) by \(A\cup\{1\}\) preserves the approximate
semigroup property, does not change \(\langle A\rangle\), and changes
\(|A\cap B_n|\) by at most one.

For \(a,b\in A\), define
\[
        \Pi(a,b)=\{c\in A:c^{-1}ab\in X\}.
\]
The condition \(A^2\subseteq AX\) is precisely the assertion that
\[
        \Pi(a,b)\neq\emptyset
        \qquad\text{for all }a,b\in A.
\]
Choose, once and for all, a map
\[
        \pi:A\times A\to A
\]
such that \(\pi(a,b)\in\Pi(a,b)\), and define
\[
        \epsilon(a,b)=\pi(a,b)^{-1}ab\in X.
\]
Thus
\[
        ab=\pi(a,b)\epsilon(a,b).
\]

We extend \(\pi\) and \(\epsilon\) to tuples in \(A^*\).  For a one-tuple \((a)\), set
\[
        \pi(a)=a,
        \qquad
        \epsilon(a)=1.
\]
For \(j\geq 2\) and
\[
        \mathbf a=(a_1,\ldots,a_j)\in A^{\times j},
\]
write \(\mathbf a'=(a_2,\ldots,a_j)\).  Assuming \(\pi(\mathbf a')\) and
\(\epsilon(\mathbf a')\) have already been defined, set
\[
        \pi(\mathbf a)=\pi(a_1,\pi(\mathbf a')),
\]
and
\[
        \epsilon(\mathbf a)
        =
        \epsilon(a_1,\pi(\mathbf a'))\epsilon(\mathbf a').
\]

\begin{prop}
\label{prop:pi}
For every tuple \(\mathbf a\in A^{\times j}\), with \(j\geq 1\), one has
\[
        \pi(\mathbf a)\in A
\]
and
\[
        \overline{\mathbf a}=\pi(\mathbf a)\epsilon(\mathbf a),
\]
where
\[
        \epsilon(\mathbf a)\in X^{j-1}.
\]
In particular,
\[
        |\epsilon(\mathbf a)|\leq (j-1)K,
\]
and hence
\[
        d(\pi(\mathbf a),\overline{\mathbf a})\leq (j-1)K.
\]
\end{prop}

\begin{proof}
The proof is by induction on \(j\).  The case \(j=1\) is immediate.  Let
\(\mathbf a=(a_1,\ldots,a_j)\) and \(\mathbf a'=(a_2,\ldots,a_j)\).  By induction,
\[
        \overline{\mathbf a'}=\pi(\mathbf a')\epsilon(\mathbf a'),
        \qquad
        \epsilon(\mathbf a')\in X^{j-2}.
\]
Therefore
\[
\begin{aligned}
        \overline{\mathbf a}
        &=a_1\overline{\mathbf a'} \\
        &=a_1\pi(\mathbf a')\epsilon(\mathbf a') \\
        &=\pi(a_1,\pi(\mathbf a'))
          \epsilon(a_1,\pi(\mathbf a'))
          \epsilon(\mathbf a') \\
        &=\pi(\mathbf a)\epsilon(\mathbf a).
\end{aligned}
\]
Since \(\epsilon(a_1,\pi(\mathbf a'))\in X\) and
\(\epsilon(\mathbf a')\in X^{j-2}\), we get
\[
        \epsilon(\mathbf a)\in X^{j-1}.
\]
The length bound follows from the definition of \(K\), and the distance estimate
follows from
\[
        \pi(\mathbf a)^{-1}\overline{\mathbf a}=\epsilon(\mathbf a).
\]
\end{proof}

\section{Proof of Theorem ~\ref{thm:hyp dichotomy}}
\label{sec:growth-ag}
\subsection{Strategy for the proof of Theorem~\ref{thm:hyp dichotomy}}
\label{subsec:strategy}
Let \(A\) be an infinite approximate semigroup in a hyperbolic group \(G\), and
suppose that \(\langle A\rangle\) is not virtually cyclic.  We want to prove that
\(A\) has positive exponential growth in the ambient word metric.

It is useful first to compare the argument with the usual ping-pong proof for
subgroups of hyperbolic groups.  If \(H\leq G\) is a non-elementary subgroup, then
one can find independent loxodromic elements and, after taking large powers, apply
ping-pong on \(\partial G\) to produce a free subgroup of rank two inside \(H\).  This
immediately gives exponential growth of \(H\).

For an approximate semigroup this is not enough.  Even if elements of
\(\langle A\rangle\) play ping-pong on the boundary, the words produced by the
ping-pong argument need not lie in \(A\).  Moreover, powers of an element of \(A\), or
of an element of \(A^2\), need not themselves belong to \(A\).  The small-doubling
relation gives a different kind of control: Proposition~\ref{prop:pi} says that a true
product of entries of \(A\) can be replaced by an element of \(A\), but only up to an
error whose word length is linear in the number of factors.

Thus the usual qualitative conclusion of ping-pong, namely that many words are
distinct, has to be strengthened to a quantitative separation statement.  We need to
construct many words whose values in the Cayley graph are separated by more than
the possible correction error coming from Proposition~\ref{prop:pi}.  The proof below
is a quantitative Schottky-type argument: we find two finite tuples in \(A\), build a
binary family of long concatenations, and prove that the corresponding group elements
are linearly separated.  After applying the approximate product map \(\pi\), the
separation survives, giving exponentially many distinct elements of \(A\) in linearly
large balls.

The following lemma isolates the geometric part of the argument.  Notice that
it does not use the approximate semigroup structure, except for the harmless
normalization \(1\in A\), which we arrange before applying it.

\begin{lemma}[Main lemma for exponential growth]
\label{lem:main}
Let \(A\subseteq G\) be an infinite subset with \(1\in A\) such that
\(\langle A\rangle\) is not virtually cyclic.  For every constant \(K_0>0\), there exist tuples
\[
        \mathbf x_0,\mathbf x_1\in A^*
\]
such that the following holds.  For \(n\geq 1\), define tuple-valued maps
\[
        \mathbf h_n,\mathbf f_n:\{0,1\}^n\to A^*
\]
by
\[
        \mathbf h_n(b)=\mathbf x_{b_1}\circ\cdots\circ\mathbf x_{b_n},
\]
and
\[
        \mathbf f_n(b)=\mathbf h_n(b)\circ \mathbf x_0^{\circ n}.
\]
Then, for every \(b\neq b'\in\{0,1\}^n\),
\[
        d(\mathbf f_n(b),\mathbf f_n(b'))
        >
        K_0\max\{\ell(\mathbf x_0),\ell(\mathbf x_1)\}n.
\]
\end{lemma}

We first deduce Theorem~\ref{thm:hyp dichotomy} from the lemma.

\begin{proof}[Proof of Theorem~\ref{thm:hyp dichotomy} assuming Lemma~\ref{lem:main}]
Let \(A\) be an infinite approximate semigroup with correction set \(X\), and assume
that \(\langle A\rangle\) is not virtually cyclic.  Enlarge \(X\), if necessary, so that
\(1\in X\), and replace \(A\) by \(A\cup\{1\}\).  This does not affect the conclusion.
Set
\[
        K=\max_{x\in X}|x|.
\]
Apply Lemma~\ref{lem:main} with \(K_0=4(K+1)\), and obtain tuples
\[
        \mathbf x_0,\mathbf x_1\in A^*.
\]
We show that \(A\) has positive exponential growth.

For \(b\in\{0,1\}^n\), the tuple \(\mathbf f_n(b)\) is a concatenation of \(2n\)
tuples, each equal to either \(\mathbf x_0\) or \(\mathbf x_1\).  Hence
\[
        \ell(\mathbf f_n(b))
        \leq
        2\max\{\ell(\mathbf x_0),\ell(\mathbf x_1)\}n,
\]
and
\[
        |\mathbf f_n(b)|
        \leq
        2\max\{|\mathbf x_0|,|\mathbf x_1|\}n.
\]
By Proposition~\ref{prop:pi},
\[
        \overline{\mathbf f_n(b)}
        =
        \pi(\mathbf f_n(b))\epsilon(\mathbf f_n(b)),
\]
with
\[
        |\epsilon(\mathbf f_n(b))|
        \leq
        K\ell(\mathbf f_n(b))
        \leq
        C_1n,
\]
where
\[
        C_1=2K\max\{\ell(\mathbf x_0),\ell(\mathbf x_1)\}.
\]
Thus
\[
        d(\pi(\mathbf f_n(b)),\overline{\mathbf f_n(b)})\leq C_1n.
\]

If \(b\neq b'\), then Lemma~\ref{lem:main} gives
\[
        d(\mathbf f_n(b),\mathbf f_n(b'))
        >
        4(K+1)\max\{\ell(\mathbf x_0),\ell(\mathbf x_1)\}n
        >
        2C_1n.
\]
Therefore
\[
\begin{aligned}
        d(\pi(\mathbf f_n(b)),\pi(\mathbf f_n(b')))
        &\geq
        d(\mathbf f_n(b),\mathbf f_n(b'))  \\
        &\quad
        -d(\pi(\mathbf f_n(b)),\overline{\mathbf f_n(b)})
        -d(\pi(\mathbf f_n(b')),\overline{\mathbf f_n(b')}) \\
        &>
        2C_1n-C_1n-C_1n
        =0.
\end{aligned}
\]
Hence
\[
        \pi(\mathbf f_n(b))\neq \pi(\mathbf f_n(b')).
\]
So
\[
        \{\pi(\mathbf f_n(b)):b\in\{0,1\}^n\}
\]
is a subset of \(A\) of cardinality \(2^n\).

Finally,
\[
\begin{aligned}
        |\pi(\mathbf f_n(b))|
        &\leq
        |\mathbf f_n(b)|+|\epsilon(\mathbf f_n(b))| \\
        &\leq
        2\max\{|\mathbf x_0|,|\mathbf x_1|\}n+C_1n
        =C_0n,
\end{aligned}
\]
where
\[
        C_0=2\max\{|\mathbf x_0|,|\mathbf x_1|\}+C_1.
\]
Thus
\[
        |A\cap B_{C_0n}|\geq 2^n
\]
for all \(n\), and \(A\) has positive exponential growth.
\end{proof}

It remains to prove Lemma~\ref{lem:main}.  We need to construct two tuples
\(\mathbf x_0,\mathbf x_1\in A^*\) such that the elements represented by the tuples
\(\mathbf f_n(b)\) are far apart whenever \(b\neq b'\).  The point of the additional
right tail \(\mathbf x_0^{\circ n}\) in the definition of \(\mathbf f_n\) is to make the
separation estimate stable after the common initial part of \(b\) and \(b'\) cancels.


Suppose for the moment that \(\mathbf x_0,\mathbf x_1\) have been chosen, and write
\[
        y_0=\overline{\mathbf x_0},
        \qquad
        y_1=\overline{\mathbf x_1},
\]
and define $\mathbf q_n(b)$ to be the tuple $(y_{b_1},\cdots y_{b_n}) \circ y_0^{\circ n}$.  Then:

\[
\overline{\mathbf f_n(b)}^{-1}\overline{\mathbf f_n(b')}=\overline{\mathbf q_n(b)}^{-1}\overline{\mathbf q_n(b')}
\]
If \(b\neq b'\), then after cancelling the common initial segment of the two binary
words, every consecutive pair that occurs in
\[
        \overline{\mathbf q_n(b)}^{-1}\overline{\mathbf q_n(b')}
\]
is one of
\begin{equation}
\label{eqn:pair list}
\begin{gathered}
(y_0^{-1},y_1^{-1}),\quad
(y_0^{-1},y_0^{-1}),\quad
(y_1^{-1},y_0^{-1}),\quad
(y_1^{-1},y_1^{-1}),\\
(y_0^{-1},y_1),\quad
(y_1^{-1},y_0),\\
(y_0,y_1),\quad
(y_0,y_0),\quad
(y_1,y_0),\quad
(y_1,y_1).
\end{gathered}
\end{equation}
Denote this finite set of ordered pairs by \(P(y_0,y_1)\).  To apply
Proposition~\ref{prop:hyp-piecewise}, it will be enough to arrange that \(|y_0|\) and
\(|y_1|\) are large and that
\[
        (u^{-1}\mid v)
\]
is uniformly bounded for every \((u,v)\in P(y_0,y_1)\).

We obtain such tuples from elements of \(A^2\) with large stable translation length.

\subsection{Elements of \texorpdfstring{$A^2$}{A²} with large stable translation length}
\label{subsec:stable length}
For a sequence \(z=(z_n)\) in \(G\), write
\[
        z^\dagger=(z_n^{-1}).
\]
We call \(z\) \emph{bi-Gromov} if both \(z\) and \(z^\dagger\) are Gromov sequences.
By applying Proposition~\ref{prop:subsequence} twice, every sequence \((z_n)\) with
\(|z_n|\to\infty\) along a subsequence has a bi-Gromov subsequence.

\begin{lemma}[Stable length from transverse limits]
\label{lem:hyp-transverse-stable}
The following hold.
\begin{enumerate}[label=\textup{(\arabic*)}]
    \item Let \(z=(z_n)\) be bi-Gromov.  If \(z\) and \(z^\dagger\) are not asymptotic,
    then
    \[
        \tau(z_n)\to\infty.
    \]
    \item Let \(x=(x_n)\) and \(y=(y_n)\) be bi-Gromov sequences.  Assume that
    \(x\) is asymptotic to \(x^\dagger\), and that \(y\) is asymptotic to
    \(y^\dagger\).  If \(x\) and \(y\) are not asymptotic, then
    \[
        \tau(x_ny_n)\to\infty.
    \]
\end{enumerate}
\end{lemma}

\begin{proof}
For \textup{(1)}, Corollary~\ref{cor:hyp-boundary-separation} gives a constant
\(M\) such that
\[
        (z_n^{-1}\mid z_n)\leq M
\]
for all \(n\).  For \(n\) with \(|z_n|>2M+2\delta\), Proposition~\ref{prop:hyp-piecewise}, applied to the tuple \((z_n)^{\circ q}\), gives
\[
        |z_n^q|
        \geq
        q|z_n|-2(M+\delta)(q-1).
\]
Dividing by \(q\) and letting \(q\to\infty\), we obtain
\[
        \tau(z_n)\geq |z_n|-2(M+\delta),
\]
which tends to infinity.

For \textup{(2)}, the sequences \(x^\dagger\) and \(y\) are not asymptotic; otherwise
\(x\) and \(y\) would be asymptotic.  Similarly, \(y^\dagger\) and \(x\) are not
asymptotic.  Hence Corollary~\ref{cor:hyp-boundary-separation} gives a constant
\(M\) such that, for all sufficiently large \(n\),
\[
        (x_n^{-1}\mid y_n)\leq M,
        \qquad
        (y_n^{-1}\mid x_n)\leq M.
\]
For large \(n\), also \(\min\{|x_n|,|y_n|\}>2M+2\delta\).  Applying
Proposition~\ref{prop:hyp-piecewise} to the \(2q\)-tuple \((x_n,y_n)^{\circ q}\), we get
\[
        |(x_ny_n)^q|
        \geq
        q(|x_n|+|y_n|)-2(M+\delta)(2q-1).
\]
After division by \(q\) and passage to the limit \(q\to\infty\),
\[
        \tau(x_ny_n)
        \geq
        |x_n|+|y_n|-4(M+\delta),
\]
which tends to infinity.
\end{proof}

\begin{lemma}[Unbounded stable length in \(A^2\)]
\label{lem:hyp-unbounded-A2}
Let \(A\subseteq G\) be infinite, assume \(1\in A\), and suppose that \(\langle A\rangle\)
is non-elementary.  Then
\[
        \sup_{h\in A^2}\tau(h)=\infty.
\]
\end{lemma}

\begin{proof}
First suppose that there is a bi-Gromov sequence \(a=(a_n)\), with \(a_n\in A\),
such that \(a\) and \(a^\dagger\) are not asymptotic.  By
Lemma~\ref{lem:hyp-transverse-stable},
\[
        \tau(a_n)\to\infty.
\]
Since \(1\in A\), we have \(A\subseteq A^2\), and the conclusion follows.

We may therefore assume that every bi-Gromov sequence in \(A\) is asymptotic to
its inverse sequence.

Suppose next that there are two Gromov sequences in \(A\), say
\[
        a=(a_n),
        \qquad
        b=(b_n),
\]
which are not asymptotic.  Passing to subsequences, we may assume that both are
bi-Gromov.  By the previous paragraph, \(a\) is asymptotic to \(a^\dagger\), and
\(b\) is asymptotic to \(b^\dagger\).  Since \(a\) and \(b\) are not asymptotic,
Lemma~\ref{lem:hyp-transverse-stable} gives
\[
        \tau(a_nb_n)\to\infty.
\]
Since \(a_nb_n\in A^2\), the conclusion follows.

We are left with the case in which all Gromov sequences in \(A\) are asymptotic.
Let \(\alpha\in\partial G\) be their common limit.  Define
\[
        B=\{a\in A:a\alpha=\alpha\}.
\]
Then
\[
        \langle B\rangle\leq \operatorname{Stab}_G(\alpha),
\]
so \(\langle B\rangle\) is elementary by Lemma~\ref{lem:hyp-elementary-facts}\textup{(1)}.
Since \(\langle A\rangle\) is non-elementary, \(A\neq B\).  Choose
\[
        a\in A\setminus B.
\]
Since \(A\) is infinite, it is unbounded.  Thus, by Proposition~\ref{prop:subsequence},
there is a bi-Gromov sequence \(t=(t_n)\) in \(A\).  By the assumptions of the current
case, both \(t\) and \(t^\dagger\) converge to \(\alpha\).  Set
\[
        r_n=at_n\in A^2.
\]
Then
\[
        r_n\to a\alpha\neq \alpha.
\]
On the other hand,
\[
        r_n^{-1}=t_n^{-1}a^{-1}.
\]
Since \(t_n^{-1}\to\alpha\), Proposition~\ref{prop:hyp-right-mult-limit} gives
\[
        r_n^{-1}\to\alpha.
\]
Thus \((r_n)\) is bi-Gromov and is not asymptotic to its inverse sequence.  By
Lemma~\ref{lem:hyp-transverse-stable},
\[
        \tau(r_n)\to\infty.
\]
Since \(r_n\in A^2\), this proves the lemma.
\end{proof}

\subsection{Proof of Lemma~\ref{lem:main}}
\begin{proof}
Since \(A\) is infinite, \(1\in A\), and \(\langle A\rangle\) is not virtually cyclic,
Lemma~\ref{lem:hyp-unbounded-A2} implies that \(\tau\) is unbounded on \(A^2\).
Choose \(g\in A^2\) such that
\begin{equation}
\label{eq:choose-g-large-tau}
        \tau(g)>3K_0,
\end{equation}
and write
\[
        g=g_1g_2
\]
with \(g_1,g_2\in A\).

Since \(\tau(g)>0\), the element \(g\) is loxodromic.  If every element of \(A\)
preserved the two-point set \(\{g^+,g^-\}\), then \(\langle A\rangle\) would preserve
this set, hence would be elementary by Lemma~\ref{lem:hyp-elementary-facts}\textup{(2)}.
This is impossible.  Therefore there exists \(a\in A\) such that
\begin{equation}
\label{eq:a-does-not-preserve-pair}
        a\{g^+,g^-\}\neq \{g^+,g^-\}.
\end{equation}
The set \(a\{g^+,g^-\}\) is the fixed-point set of the conjugate \(aga^{-1}\).  By
Lemma~\ref{lem:lox}\textup{(\ref{shared-fixed-point})}, \eqref{eq:a-does-not-preserve-pair} implies
\begin{equation}
\label{eq:agplus-not-endpoints}
        a\{g^+,g^-\}\cap \{g^+,g^-\}=\emptyset.
\end{equation}

For \(N\geq 1\), define tuples
\[
        \mathbf u_N=(g_1,g_2)^{\circ N},
        \qquad
        \mathbf v_N=(a)\circ(g_1,g_2)^{\circ N}.
\]
Then
\[
        \ell(\mathbf u_N)=2N,
        \qquad
        \ell(\mathbf v_N)=2N+1.
\]
Set
\begin{equation}
\label{eqn:yN and zN}
        y_N=\overline{\mathbf u_N}=g^N,
        \qquad
        z_N=\overline{\mathbf v_N}=ag^N.
\end{equation}
By Proposition~\ref{prop:tau},
\[
        |y_N|\geq N\tau(g),
        \qquad
        |z_N|\geq N\tau(g)-|a|.
\]
Moreover,
\[
        y_N\to g^+,
        \qquad
        y_N^{-1}\to g^-,
\]
and
\[
        z_N\to ag^+.
\]
Also
\[
        z_N^{-1}=g^{-N}a^{-1}\to g^-
\]
by Proposition~\ref{prop:hyp-right-mult-limit}.  By \eqref{eq:agplus-not-endpoints}, among the four sequences
\[
        (y_N),\quad (y_N^{-1}),\quad (z_N),\quad (z_N^{-1}),
\]
the only asymptotic pair is \((y_N^{-1})\) and \((z_N^{-1})\).

We claim that there is a constant \(M_0\geq 0\) such that, for every \(N\) and every
\((s_N,t_N)\in P(y_N,z_N)\),
\[
        (s_N^{-1}\mid t_N)\leq M_0.
\]

Indeed, since $(y_N,z_N^{-1}),(z_n,y_N^{-1}) \not\in P(y_N,Z)n)$,  for each $((s_N,t_N)\in P(y_N,z_N)\),  the two sequences
\((s_N^{-1})\) and \((t_N)\) are not asymptotic. The existence of a uniform
\(M_0\) follows from
Corollary~\ref{cor:hyp-boundary-separation}, since \(P(y_N,z_N)\) has only finitely
many pairs.

Choose an integer \(m\ge 1\) such that
\[
        m\tau(g)>4M_0+4\delta+2|a|.
\]
Then
\[
        m\tau(g)-|a|-2(M_0+\delta)>\frac12 m\tau(g).
\]
In particular, \(|y_m|\) and \(|z_m|\) are both larger than \(2M_0+2\delta\).

Now define
\[
        \mathbf x_0=\mathbf u_m,
        \qquad
        \mathbf x_1=\mathbf v_m.
\]
We claim that these tuples satisfy Lemma~\ref{lem:main}.

Let \(b\neq b'\in\{0,1\}^n\).  Let \(p\) be their longest common prefix, and write
\[
        b=pc,
        \qquad
        b'=pc',
\]
where \(c\) and \(c'\) have the same length \(t\), and their first letters are different.
For a binary word \(c\), let \(\mathbf q(c)\) be the tuple over \(G\) obtained by
replacing each \(0\) by \(y_m\) and each \(1\) by \(z_m\).  After the common prefix
has cancelled, we get
\[
        \overline{\mathbf f_n(b)}^{-1}\overline{\mathbf f_n(b')}
        =
        \overline{\mathbf r},
\]
where
\[
        \mathbf r
        =
        (y_m^{-1})^{\circ n}
        \circ
        \mathbf q(c)^{-1}
        \circ
        \mathbf q(c')
        \circ
        (y_m)^{\circ n}.
\]
The tuple \(\mathbf r\) has length \(2n+2t\).  Every entry is one of
\[
        y_m,
        \quad y_m^{-1},
        \quad z_m,
        \quad z_m^{-1},
\]
and each has length at least \(m\tau(g)-|a|\).  Every consecutive pair appearing in
\(\mathbf r\) belongs to \(P(y_m,z_m)\), so for each consecutive pair \((u,v)\) in
\(\mathbf r\),
\[
        (u^{-1}\mid v)\leq M_0.
\]
By Proposition~\ref{prop:hyp-piecewise},
\[
\begin{aligned}
        |\mathbf r|
        &\geq
        (2n+2t)(m\tau(g)-|a|)-2(2n+2t-1)(M_0+\delta) \\
        &\geq
        (2n+2t)\bigl(m\tau(g)-|a|-2(M_0+\delta)\bigr).
\end{aligned}
\]
By the choice of \(m\),
\[
        m\tau(g)-|a|-2(M_0+\delta)>\frac12m\tau(g).
\]
Therefore
\[
        |\mathbf r|
        >
        (n+t)m\tau(g)
        \geq
        nm\tau(g)
        >
        3K_0mn.
\]
Since \(m\geq 1\),
\[
        3K_0mn
        \geq
        K_0(2m+1)n
        =
        K_0\max\{\ell(\mathbf x_0),\ell(\mathbf x_1)\}n.
\]
Finally,
\[
        d(\mathbf f_n(b),\mathbf f_n(b'))
        =
        |\overline{\mathbf f_n(b)}^{-1}\overline{\mathbf f_n(b')}|
        =
        |\mathbf r|
        >
        K_0\max\{\ell(\mathbf x_0),\ell(\mathbf x_1)\}n.
\]
This proves Lemma~\ref{lem:main}, and hence Theorem~\ref{thm:hyp dichotomy}.
\end{proof}

\section{Well-defined growth rates of approximate groups}\label{sec:growth-rates}
In this section we prove the general growth-rate criterion used in the introduction. 
The point is that a polynomial lower bound for ordinary product sets gives the
almost-superadditivity needed for the counting function of an approximate semigroup.
More precisely, we show that if the ambient finitely generated group satisfies large
product growth, then every approximate semigroup \(A\subseteq G\) has a well-defined
exponential growth rate,
\[
        \lim_{n\to\infty}\frac{1}{n}\log |A\cap B_n|.
\]
The proof is based on a shifted almost-superadditivity estimate, reduced to the
classical theorem of de Bruijn and Erd\H{o}s.

\begin{definition}[Product shrinkage]\label{dfn:product-shrinkage}
Let \(G\) be a finitely generated group with a fixed finite generating set, and let
\(B_n\) denote the ball of radius \(n\) around the identity.  Let
\(f:\mathbb N\to [1,\infty)\) be a non-decreasing function.  We say that \(G\) has
\emph{product shrinkage at most \(f\)} if, for all finite subsets
\(U\subseteq B_n\) and \(V\subseteq B_k\),
\[
        |UV|\ge \frac{|U|\,|V|}{f(n+k)}.
\]
\end{definition}

Changing the finite generating set changes word lengths only up to a multiplicative
constant.  Thus a shrinkage bound \(f\) is replaced by a bound of the form \(f(Cn)\)
for some constant \(C\).  In particular, polynomial product shrinkage is independent
of the choice of finite generating set.

A natural problem is to determine the optimal shrinkage function for specific groups.
For instance, given a finitely generated group \(G\), one may ask for the smallest
possible degree of a polynomial \(P\) such that
\[
        |UV|\ge \frac{|U|\,|V|}{P(n+k)}
\]
for all finite sets \(U\subseteq B_n\) and \(V\subseteq B_k\).  This degree is independent
of the choice of finite generating set.  The following theorem shows more generally
that any shrinkage function whose logarithm satisfies the summability condition below
is enough to force the existence of growth rates for approximate semigroups.

\begin{theorem}\label{thrm:existence-of-limit}
Let \(G\) be a finitely generated group with product shrinkage at most \(f\), where
\(f:\mathbb N\to [1,\infty)\) is non-decreasing and satisfies
\[
        \sum_{n=1}^{\infty}\frac{\log f(n)}{n^2}<\infty.
\]
Then, for every non-empty approximate semigroup \(A\subseteq G\), the limit
\[
        \lim_{n\to\infty}\frac{1}{n}\log |A\cap B_n|
\]
exists in \(\mathbb R\).
\end{theorem}

For example, product shrinkage at most
\[
        f(n)=c^{n^{1-\epsilon}}
\]
with \(c>1\) and \(\epsilon>0\) is sufficient to apply the theorem.

To prove Theorem~\ref{thrm:existence-of-limit}, we use the following shifted version
of the de Bruijn--Erd\H{o}s almost-superadditivity theorem \cite{de1951some} (which is the theorem below with $C_2=0$).

\begin{lemma}[Shifted almost superadditivity]\label{lem:shifted_nearly_fekete}
Let \(f:\mathbb N\to\mathbb R_{\ge 0}\) be non-decreasing, and assume there is a constant \(C_1>0\) such that
\[
f(n)\le C_1 n \qquad \forall n\in\mathbb N.
\]
Assume there exist \(C_2\in\mathbb N\) and a function \(g:\mathbb N\to\mathbb R_{\ge 0}\) such that
\begin{equation}\label{eq:shifted_nearly_superadd}
f(x+y+C_2)\ge f(x)+f(y)-g(x+y)\qquad \forall x,y\in\mathbb N.
\end{equation}
Let
\[
G(n):=\max_{1\le k\le n} g(k),
\]
and assume that
\begin{equation}\label{eq:DBE_G}
\sum_{n=1}^\infty \frac{G(n)}{n^2}<\infty.
\end{equation}
Then the limit
\[
\lim_{n\to\infty}\frac{f(n)}{n}
\]
exists in \(\mathbb R\).
\end{lemma}

\begin{remark}
Lemma~\ref{lem:shifted_nearly_fekete} remains true if \(f\) is allowed to take arbitrary real values.
\end{remark}

\begin{proof}
Define \(h:\mathbb N\to\mathbb R_{\geq 0}\) by
\[
        h(n):=
        \begin{cases}
        0,& n<C_2,\\
        f(n-C_2),& n\ge C_2.
        \end{cases}
\]
Then \(h\) is non-decreasing and satisfies \(h(n)\le C_1n\) for all sufficiently
large \(n\), hence after changing \(C_1\) it satisfies \(h(n)\le C_1n\) for all
\(n\).

We claim that \(h\) satisfies an unshifted almost-superadditivity inequality.  If
\(u,v\ge C_2\), then
\[
\begin{aligned}
h(u+v)
&= f(u+v-C_2) \\
&\ge f(u-C_2)+f(v-C_2)-g(u+v-2C_2) \\
&= h(u)+h(v)-g(u+v-2C_2).
\end{aligned}
\]
If at least one of \(u,v\) is less than \(C_2\), then the corresponding value of
\(h\) is \(0\), and since \(h\) is non-decreasing we have
\[
        h(u+v)\ge h(u)+h(v).
\]
Thus, defining
\[
        \widetilde g(t):=
        \begin{cases}
        0,& t<2C_2,\\
        g(t-2C_2),& t\ge 2C_2,
        \end{cases}
\]
we have
\[
        h(u+v)\ge h(u)+h(v)-\widetilde g(u+v)
\]
for all \(u,v\in\mathbb N\).

Let
\[
        \widetilde G(n):=\max_{1\leq t\leq n}\widetilde g(t).
\]
Then \(\widetilde G(n)\le G(n)\) for all sufficiently large \(n\), and therefore
\[
        \sum_{n=1}^{\infty}\frac{\widetilde G(n)}{n^2}<\infty.
\]
Hence the classical theorem of de Bruijn and Erd\H{o}s (which is the present lemma with $C_2=0$) applies to \(h\), and gives
the existence of
\[
        \lim_{n\to\infty}\frac{h(n)}{n}.
\]

Since \(h(n)=f(n-C_2)\), this is equivalent to the existence of
\[
\lim_{n\to\infty}\frac{f(n)}{n}.
\]
\end{proof}

We are now ready to prove Theorem~\ref{thrm:existence-of-limit}.

\begin{proof}[Proof of Theorem~\ref{thrm:existence-of-limit}]
Let \(A\subseteq G\) be a non-empty approximate semigroup.  If \(A\) is finite, then
\[
        \lim_{n\to\infty}\frac{1}{n}\log |A\cap B_n|=0,
\]
so we may assume that \(A\) is infinite.

Choose a finite set \(X\subseteq G\) such that
\[
        A^2\subseteq AX.
\]
Enlarging \(X\), and replacing \(A\) by \(A\cup\{1\}\) if necessary, we may assume
that \(1\in X\) and \(1\in A\).  This replacement does not change the desired limit:
the two counting functions differ by at most \(1\), and \(|A\cap B_n|\to\infty\).
Set
\[
        K:=\max_{x\in X}|x|.
\]

For \(n\in\mathbb N\), define
\[
        A_n:=A\cap B_n,
        \qquad
        a(n):=\log |A_n|.
\]

Since \(1\in A\), each \(A_n\) is nonempty, so \(a\) is well-defined.
Moreover, \(a\) is non-decreasing and takes values in \(\mathbb R_{\ge 0}\).

Since \(A_n\subseteq B_n\), we have
\[
a(n)\le \log |B_n|.
\]
Because \(G\) is finitely generated, balls grow at most exponentially, so there exists \(C_1>0\) such that
\[
a(n)\le C_1 n
\qquad \forall n\in\mathbb N.
\]

We now prove the shifted almost-superadditivity estimate.
Let \(n,m\in\mathbb N\).
Since \(A_n,A_m\subseteq A\), we have
\[
        A_nA_m\subseteq A^2\subseteq AX.
\]
Also, every element \(w\in A_nA_m\) satisfies
\[
        |w|\le n+m.
\]
If \(w\in A_nA_m\), write \(w=ax\) with \(a\in A\) and \(x\in X\).  Then
\[
        |a|=|wx^{-1}|\le |w|+|x|\le n+m+K.
\]
Hence \(a\in A_{n+m+K}\), so
\[
        A_nA_m\subseteq A_{n+m+K}X.
\]
Therefore
\[
        |A_nA_m|\le |A_{n+m+K}|\,|X|,
\]
or equivalently,
\begin{equation}\label{eq:skewd-set-growth}
|A_{n+m+K}|\ge \frac{|A_nA_m|}{|X|}.
\end{equation}

Since \(G\) has product shrinkage at most \(f\), we have
\[
        |A_nA_m|\ge \frac{|A_n|\,|A_m|}{f(n+m)}.
\]
Combining this with \eqref{eq:skewd-set-growth}, we obtain
\[
        |A_{n+m+K}|
        \ge
        \frac{|A_n|\,|A_m|}{|X|\,f(n+m)}.
\]
Taking logarithms gives
\[
        a(n+m+K)
        \ge
        a(n)+a(m)-\bigl(\log |X|+\log f(n+m)\bigr).
\]
Thus Lemma~\ref{lem:shifted_nearly_fekete} applies with
\[
        C_2=K,
        \qquad
        g(t)=\log |X|+\log f(t).
\]
Indeed, since \(f\) is non-decreasing, the function \(g\) is non-decreasing, and
\[
        \sum_{t=1}^{\infty}\frac{g(t)}{t^2}
        =
        \log|X|\sum_{t=1}^{\infty}\frac1{t^2}
        +
        \sum_{t=1}^{\infty}\frac{\log f(t)}{t^2}
        <\infty.
\]

All the assumptions of Lemma~\ref{lem:shifted_nearly_fekete} are satisfied. We conclude that the limit
\[
        \lim_{n\to\infty}\frac{a(n)}{n}
\]
exists. Since \(a(n)=\log |A\cap B_n|\), this is precisely the existence of
\[
        \lim_{n\to\infty}\frac{1}{n}\log |A\cap B_n|.
\]
\end{proof}

\subsection{Large product growth}
\label{subsec:large-product-growth}

Theorem~\ref{thrm:existence-of-limit} applies to any group whose product shrinkage
function satisfies the summability condition above. As we will see in section \ref{sec:rd-product-growth}, the property of \emph{rapid decay} implies the much stronger condition of polynomial product shrinkage, which we highlight with the following definition.

\begin{definition}[Large product growth]\label{dfn:large-product-growth}
Let \(G\) be a finitely generated group with a fixed finite generating set.  We say
that \(G\) has \emph{large product growth} if it has product shrinkage at most \(P\)
for some polynomial \(P\).
\end{definition}

If \(G\) has large product growth, then the hypothesis of
Theorem~\ref{thrm:existence-of-limit} is satisfied.  Indeed, for a polynomial \(P\),
\[
        \log P(n)=O(\log n),
\]
and hence
\[
        \sum_{n=1}^{\infty}\frac{\log P(n)}{n^2}<\infty.
\]

We also record the following consequence in the polynomial-shrinkage case: if the
exponential growth rate is zero, then the growth is polynomially bounded.

\begin{corollary}
\label{cor:zero-growth-polynomial}
Let \(G\) be a finitely generated group with large product growth, and let
\(A\subseteq G\) be a non-empty approximate semigroup.  If
\[
        \lim_{n\to\infty}\frac{1}{n}\log |A\cap B_n|=0,
\]
then \(|A\cap B_n|\) is polynomially bounded.
\end{corollary}

\begin{proof}
Use the notation from the proof of Theorem~\ref{thrm:existence-of-limit}; in
particular,
\[
        a(n)=\log |A\cap B_n|.
\]
Since \(G\) has large product growth, we may take the shrinkage function to be
polynomial.  Therefore, from the proof, there are constants \(K\geq 0\) and \(C>0\)
such that
\[
        a(n+m+K)\geq a(n)+a(m)-C\log(1+n+m)
\]
for all \(n,m\geq 1\).

Fix \(N\geq 1\), and define
\[
        N_0=N,\qquad N_{r+1}=2N_r+K.
\]
Then \(N_r\leq 2^r(N+K)\).  Iterating the displayed inequality gives
\[
        \frac{a(N_r)}{2^r}
        \geq
        a(N)-C\sum_{i=0}^{r-1}2^{-i-1}\log(1+2N_i+K).
\]
Since \(N_i\leq 2^i(N+K)\), the sum is bounded by
\[
        C'\log(1+N+K)
\]
for a constant \(C'\) independent of \(N\) and \(r\).  Hence
\[
        a(N_r)\geq 2^r\bigl(a(N)-C'\log(1+N+K)\bigr)
\]
If the quantity in parentheses were positive for some \(N\), then
\[
        \limsup_{r\to\infty}\frac{a(N_r)}{N_r}>0,
\]
contradicting the assumption that the exponential growth rate is zero.  Therefore
\[
        a(N)\leq C'\log(1+N+K)
\]
for all \(N\).  Thus
\[
        |A\cap B_N|\leq (1+N+K)^{C'},
\]
so \(|A\cap B_N|\) is polynomially bounded.
\end{proof}

\subsection{A sufficient condition for linear product shrinkage}
\label{subsec:sufficient-linear-shrinkage}
For groups with large product growth, a natural question is to determine
the optimal degree of the polynomial shrinkage factor.
In this subsection we record a simple sufficient condition for linear product shrinkage.  This isolates
an abstract layer-peeling argument that will be used below for both free groups and
hyperbolic groups.

Let
\[
        S_i:=\{g\in G:|g|=i\}.
\]
For a finite set \(U\subseteq G\), write
\[
        U_i:=U\cap S_i,
        \qquad
        \Lambda^+(U):=\{i\geq 1:U_i\neq\emptyset\}.
\]

\begin{definition}[Property \(Q(\zeta)\)]\label{dfn:property-Q}
Let \(0<\zeta\leq 1\).  We say that \(G\) satisfies \emph{property \(Q(\zeta)\)} if,
for all finite non-empty sets \(U,V\subseteq G\) with both \(\Lambda^+(U)\) and
\(\Lambda^+(V)\) non-empty, either there exists \(i\in\Lambda^+(U)\) such that
\[
        |U_iV|\geq \zeta |U_i||V|,
\]
or there exists \(j\in\Lambda^+(V)\) such that
\[
        |UV_j|\geq \zeta |U||V_j|.
\]
\end{definition}

\begin{theorem}[Layer peeling from property \(Q(\zeta)\)]\label{thm:Q-layer-peeling}
Suppose that \(G\) satisfies property \(Q(\zeta)\).  Let \(U,V\subseteq G\) be finite
non-empty sets and assume that
\[
        r(U,V):=|\Lambda^+(U)|+|\Lambda^+(V)|>0.
\]
Then
\[
        |UV|\geq \frac{\zeta}{r(U,V)}|U||V|.
\]
\end{theorem}

\begin{proof}
We prove a decomposition claim.  If \(r(U,V)>0\), then \(U\times V\) can be
partitioned into at most \(r(U,V)\) rectangles
\[
        A_1\times C_1,\ldots,A_m\times C_m,
        \qquad m\leq r(U,V),
\]
such that \(A_s\subseteq U\), \(C_s\subseteq V\), and
\[
        |A_sC_s|\geq \zeta |A_s||C_s|
\]
for every \(s\).

We prove the claim by induction on \(r(U,V)\).  If \(U\subseteq\{1\}\), then the
single rectangle \(U\times V\) works, since \(U=\{1\}\) and
\(|UV|=|V|=|U||V|\).  The case \(V\subseteq\{1\}\) is identical.

Now assume that both \(\Lambda^+(U)\) and \(\Lambda^+(V)\) are non-empty.  By
property \(Q(\zeta)\), either there exists \(i\in\Lambda^+(U)\) such that
\[
        |U_iV|\geq \zeta |U_i||V|,
\]
or there exists \(j\in\Lambda^+(V)\) such that
\[
        |UV_j|\geq \zeta |U||V_j|.
\]
In the first case, set \(U'=U\setminus U_i\).  If \(U'=\emptyset\), then
\(U\times V=U_i\times V\) is a good rectangle.  If \(U'\neq\emptyset\), then
\(r(U',V)=r(U,V)-1\), so by induction \(U'\times V\) can be partitioned into at
most \(r(U,V)-1\) good rectangles.  Adding the rectangle \(U_i\times V\) gives the
desired partition of \(U\times V\).  The second case is identical, removing \(V_j\)
instead.  This proves the claim.

Using the claim, choose such a partition.  Since the rectangles partition \(U\times V\),
\[
        \sum_{s=1}^m |A_s||C_s|=|U||V|.
\]
Hence for some \(s\),
\[
        |A_s||C_s|\geq \frac{|U||V|}{m}
        \geq \frac{|U||V|}{r(U,V)}.
\]
Since \(A_sC_s\subseteq UV\), we get
\[
        |UV|\geq |A_sC_s|
        \geq \zeta |A_s||C_s|
        \geq \frac{\zeta}{r(U,V)}|U||V|.
\]
\end{proof}

\begin{corollary}\label{cor:Q-linear-product-shrinkage}
Suppose that \(G\) satisfies property \(Q(\zeta)\).  If \(U\subseteq B_n\) and
\(V\subseteq B_k\) are finite non-empty sets, then
\[
        |UV|\geq \frac{\zeta}{n+k+1}|U||V|.
\]
\end{corollary}

\begin{proof}
If \(|\Lambda^+(U)|+|\Lambda^+(V)|>0\), then
Theorem~\ref{thm:Q-layer-peeling} gives
\[
        |UV|\geq
        \frac{\zeta}{|\Lambda^+(U)|+|\Lambda^+(V)|}|U||V|.
\]
Since \(|\Lambda^+(U)|\leq n\) and \(|\Lambda^+(V)|\leq k\), this implies the
claimed bound.  If \(|\Lambda^+(U)|+|\Lambda^+(V)|=0\), then
\(U=V=\{1\}\), and the estimate is trivial.
\end{proof}

The remaining sections give several sources of large product growth and also study the product-shrinkage problem in its own right.  
The implication from Rapid Decay to large product growth is a standard consequence of the usual convolution estimates,
applied to indicator functions together with the Cauchy--Schwarz energy inequality; we include the proof in the final section for completeness.  
For the existence of the limit in Theorem~\ref{thrm:existence-of-limit}, any polynomial shrinkage bound is sufficient, so sharper estimates do not change the conclusion once such a bound is
known.  
Nevertheless, determining the optimal shrinkage function for a given ambient group is a separate and natural quantitative problem.  
Rapid Decay generally gives only a polynomial bound and does not identify the optimal dependence on the radii.
The direct arguments in the next sections address this sharper question: for free and hyperbolic groups they give the optimal linear order of loss in the ball--ball regime,
and in free groups they also determine the exact best constant for products of subsets
of spheres.


\section{Large product growth for free groups}

In this section we prove large product growth for free groups. 
More precisely, we show that for every free group \(F_r\) with \(r\ge 2\), product sets of finite subsets of balls have at most shrinkage linear in the sum of the diameters
of the containing balls.
Our bound is tight up to a constant factor (of $1/2$).

We then refine this estimate in the case where the two
sets are supported on spheres, and completely solve the problem, by presenting matching upper and lower bounds. 

These results provide the free-group case of the large product growth condition introduced
in the previous section.

For words $u,v \in F_r$, the \emph{cancellation length} of $u,v$,
is the length of the maximal word $w$ such that $w$ is a prefix of $v$ and
$w^{-1}$ is a suffix of $u$ and is equal to 
\[
        \frac{|u|+|v|-|uv|}{2}=(u^{-1}\mid v).
\]

\subsection{A ball--ball product estimate in free groups}

Let \(F_r\) be the free group of rank \(r\ge 2\), equipped with the word metric
with respect to a fixed free generating set. We write
\[
S_m=\{g\in F_r: |g|=m\},\qquad B_m=\{g\in F_r: |g|\le m\}.
\]
For  \(U\subseteq F_r\), let
\[
U_i=U\cap S_i.
\]
Let $\Lambda(U)=\{i:U_i \neq \emptyset\}$ and $\Lambda^+(U)=\Lambda(U) \cap [1,\infty)$.

We first record the basic injectivity claim which will be used throughout.

\begin{proposition}[Marker injectivity]\label{prop:marker-injectivity}
Let  $U,V \subset F_r$ with \(V\subseteq S_m\).  Suppose $p$ is a word that is a prefix of every word in $V$ and let $P$ be the set of pairs $(u,v) \in U \times V$ such that $(u^{-1}\mid v)\leq |p|$.
Then
the multiplication map
\[
P\longrightarrow F_r,\qquad (u,v)\mapsto uv
\]
is injective and therefore $|UV|\geq |P|$.
\end{proposition}

\begin{proof}
Let $p=p_1p_2\ldots p_{j}$ be the expression of $p$ as a product
of generators.
Suppose $(u,v),(u',v') \in P$ and $uv=u'v'$. We prove $(u,v)=(u',v')$.

Let \(c,c'\le j\) be the cancellation
lengths of $u,v$ and $u',v'$, respectively. We can therefore
define the prefix $a$ of $u$ and suffix $b$ of $v$ such that:
\[
u=a(p_1\cdots p_c)^{-1},\qquad
v=(p_1\cdots p_{j})b,
\]
and so that \(uv=ap_{c+1}\cdots p_{j}b\) is reduced. Similarly,
\[
u'=a'(p_1\cdots p_{c'})^{-1},\qquad
v'=(p_1\cdots p_{j})b',
\]
so that \(u'v'=a'p_{c'+1}\cdots p_{j} b'\) is reduced.
Thus $ap_{c+1}\cdots p_{j}b =a'p_{c'+1}\cdots p_{j}b'$.
and both of these representations are reduced. Since $|b|=|b'|=m-j$, $b=b'$
and therefore $v=v'$ and since $uv=u'v'$ we also have $u=u'$.
\end{proof}

We use this Proposition to prove:

\begin{lemma}[Heavy-prefix/suffix lemma]\label{lem:heavy-prefix}
The free group satisfies Property $Q(\zeta)$ with $\zeta = 1/4$. 
That is, let \(U,V\subseteq F_r\) be finite with both $\Lambda^+(U)$ and $\Lambda^+(V)$ nonempty.  Then either there
exists \(i \in \Lambda^+(U)\) such that
\[
|U_iV|\ge \frac14 |U_i||V|,
\]
or there exists \(j \in \Lambda^+(V)\) such that
\[
|UV_j|\ge \frac14 |U||V_j|.
\]
\end{lemma}

\begin{proof}
For a word $w$, let $V(w)$ be the set of words of $V$ that
begin with $w$ and $V_j(w)=V(w) \cap S_j$. For $j \in \Lambda^+(V)$,
let $\alpha_j(w)=|V_j(w)|/|V_j|$ (the fraction of words 
in $V_j$ that have $w$ as a prefix) and for $j \in \Lambda^+(U)$ let $\beta_j(w)$ denote
the fraction of words of $U_j$ that have $w^{-1}$ as a suffix.
Say that a word $w$ is \emph{half-heavy for $V_j$} if $\alpha_j(w) \geq 1/2$
is \emph{half-heavy for $U_j$} if $\beta_j(w) \geq 1/2$.

If there is no non-empty word that is half-heavy for any layer
then choose any $j \in \Lambda^+(V)$. We claim $|UV_j|\geq \frac{1}{2}|U||V_j|$.
Let \(P\) be the set of pairs \((u,v) \in U \times V_j\) with cancellation length 0.
Consider an arbitrary \(u \in U\).  If \(u=1\), then \((u,v)\in P\) for every
\(v\in V_j\).  Otherwise, let \(a\) be the final letter of \(u\).  Since
\(\alpha_j(a^{-1})<1/2\), the pair \((u,v)\) lies in \(P\) for more than half of the \(v \in V_j\).  Therefore \(|P|\geq \frac12|U||V_j|\), and by
Proposition~\ref{prop:marker-injectivity}, with \(p\) the empty word, we conclude that
\[
        |UV_j|\geq \frac12 |U||V_j|.
\]

So assume that there is a non-trivial word that is half-heavy for some
layer. Let $\ell$ be the length of the longest such word and among
words of length $\ell$ choose a word $w$ for which $\max(\max_{j \in \Lambda^+(V)} \alpha_j(w), \max_{j \in \Lambda^+(U)} \beta_j(w))$ is maximum.
Assume that this maximum occurs for level $V_j$ (the case that it occurs
for $U_i$ is similar.)  We have $|V_j(w)|=\alpha_j(w)|V_j|$.

Now let $P$ be the set of pairs $(u,v)\in U \times V_j(w)$
with cancellation length at most $\ell$.  By Proposition
~\ref{prop:marker-injectivity}, $|UV_j(w)|\geq |P|$.  We claim
$|UV_j(w)| \geq \frac{1}{2}|U||V_j(w)|$. Fix
$v \in V_j(w)$ and consider the set $U'=\{u \in U:(u,v) \in P\}$
and let $U'_i=U'\cap U_i$.  We claim that $|U'_i| \geq \frac{1}{2}|U_i|$.
This is trivially true for $i=0$.  For $i \in \Lambda^+(U)$,
if $|v|=\ell$ then $U'_i=U_i$.
If $|v|>\ell$, let $a$ be the letter in position $\ell+1$ of $v$.
$(u,v)$ has cancellation length greater than $\ell$ if and only if $(wa)^{-1}$ is a suffix of $u$ and thus $|U'_i=(1-\beta_i(wa))|U_i|$.  
By the choice of $w$, $\beta_i(wa)<1/2$, so $|U'_i|\geq|U_i|/2$.  
Since this is true for all $i$, $|P| \geq \frac{1}{2}|U||V_j(w)|\geq \frac{\alpha_j(w)}{2}|U||V_j| \geq \frac{1}{4}|U||V_j|$.

\end{proof}


We now deduce the ball--ball estimate from the general layer-peeling criterion of
Section~\ref{subsec:sufficient-linear-shrinkage}.  Lemma~\ref{lem:heavy-prefix}
says precisely that the free group satisfies property \(Q(1/4)\).

\begin{theorem}[Ball--ball estimate]\label{thm:ball-ball-heavy-prefix}
Let \(U,V\subseteq F_r\) be finite non-empty sets, and suppose
\[
        |\Lambda^+(U)|+|\Lambda^+(V)|>0.
\]
Then
\[
        |UV|\ge
        \frac{|U||V|}
        {4\bigl(|\Lambda^+(U)|+|\Lambda^+(V)|\bigr)}.
\]
\end{theorem}

\begin{proof}
This is Theorem~\ref{thm:Q-layer-peeling} applied with \(\zeta=1/4\).
\end{proof}

\begin{corollary}
Let \(U\subseteq B_n\) and \(V\subseteq B_k\) be finite non-empty sets. Then
\[
        |UV|\ge \frac{|U||V|}{4(n+k)+1}.
\]
\end{corollary}

\begin{proof}
This follows from Corollary~\ref{cor:Q-linear-product-shrinkage} with \(\zeta=1/4\).
\end{proof}

\begin{remark}
The linear dependence on the radius is optimal in order.  Indeed, take a nontrivial element \(a\in F_r\) and consider
\[
        U=V=\{a^{-N},a^{-N+1},\ldots,a^N\}\subseteq B_N.
\]
Then
\[
        |U|=|V|=2N+1,
        \qquad
        |UV|=4N+1.
\]
Thus any bound of the form
\[
        |UV|\geq \frac{|U||V|}{f(n+k)}
\]
requires, when \(n=k=N\), that
\[
        f(2N)\geq \frac{(2N+1)^2}{4N+1},
\]
which is asymptotic to \(N\).
\end{remark}

\subsection{A sharp bound for products of subsets of spheres}

We now prove a sharp lower bound on the ratio
$\frac{|UV|}{|U||V|}$ when each of $U$ and $V$ is a subset
of a sphere.

Define $\gamma(n,k)=\min_{U \subseteq S_n,V \subseteq S_k} \frac{|UV|}{|U||V|}$.

Let $(d_j:j \geq 0)$ denote the sequence given by the recurrence $d_0=1$ and
$d_j=\frac{1}{2}+\frac{d_{j-1}}{4}$.  Thus $d_j=\frac{2+4^{-j}}{3}$,

\begin{theorem}[Sphere--sphere product estimate]\label{thm:sphere-sphere}
For all $n \geq k \geq 0$, $$\gamma(n,k)=\gamma(k,n)=d_k.$$
\end{theorem}

We first note an elementary inequality.

\begin{proposition}\label{prop:optimization}
Let $(\lambda_i:i \in I)$ be a probability distribution on finite set
 \(I\) and let $m$ be an index for which $\lambda_m$ is maximum.
 Then for any constant $c \in [0,1/3]$,
 
\[
\sum_{i \neq m}\lambda_i^2+c\lambda_m^2 \leq \frac{c+1}{4}.
\]
\end{proposition}
\begin{proof}
If $\lambda_m \leq \frac{1}{2}$ then the sum is at most 
\begin{eqnarray*}
    \lambda_m(\sum_{i \neq m}\lambda_i + c\lambda_m) & \leq & \lambda_m(1-\lambda_m+c\lambda_m)\\
    &=& \lambda_m-(1-c)\lambda_m^2 \leq \frac{1+c}{4},
\end{eqnarray*}
where the final inequality holds because the previous expression is an increasing function of $\lambda_m$ on $[0,\frac{1}{2}]$.

If $\lambda_m > \frac{1}{2}$, the sum is at most:

\begin{eqnarray*}
    \sum_{i \neq m}\lambda_i^2 + c\lambda_m^2 & \leq & (\sum_{i \neq m} \lambda_i)^2+c\lambda_m^2\\
    &=& (1-\lambda_m)^2+c\lambda_m^2 \leq \frac{1+c}{4},
\end{eqnarray*}
where the final inequality holds because the previous expression is convex
in $\lambda_m$ and so is maximized either at 1/2 or 1. Since $c\leq 1/3$ it is maximized at $\lambda_m=1/2$. 
\end{proof}

\begin{proof}[Proof of Theorem~\ref{thm:sphere-sphere}]

By symmetry we have $\gamma(n,k)=\gamma(k,n)$ so it suffices
to prove that for $n \geq k$, $\gamma(n,k)=d_k$.

First we prove that $\gamma(n,k) \geq d_k$.
We use induction on \(k\). Assume $U \subseteq S_n$
and $V \subseteq S_k$; we show that $|UV|\geq |U||V|d_k$.

The case \(k=0\) is trivial,
since \(V\subseteq S_0=\{1\}\), so \(|UV|=|U||V|\).

Assume now \(k\ge 1\), and suppose the $\gamma(n,k-1)=d_{k-1}$
for $n \geq k-1$.

For each oriented generator \(a\in\mathcal A\), define
\[
 U(a)=\{u\in U:\text{ the last letter of }u\text{ is }a^{-1}\},
\]
and
\[
 V(a)=\{v\in V:\text{ the first letter of }v\text{ is }a\}.
\]
Thus
\[
 U=\bigsqcup_{a\in\mathcal A}U(a),
 \qquad
 V=\bigsqcup_{a\in\mathcal A}V(a).
\]
Write
\[
 \alpha_a=\frac{|U(a)|}{|U|},
 \qquad
 \beta_a=\frac{|V(a)|}{|V|}.
\]
These are probability vectors.

If \(a\neq b\), then every product in \(U(a)V(b)\) has no cancellation at the
junction. Hence every element of \(U(a)V(b)\) has length \(n+k\), and the splitting
after the first \(n\) letters recovers the pair \((u,v)\). Therefore the sets
\(U(a)V(b)\), with \(a\neq b\), are pairwise disjoint and
\[
 \left|\bigcup_{a\neq b}U(a)V(b)\right|
 =
 \sum_{a\neq b}|U(a)||V(b)|
 =
 |U||V|\left(1-\sum_a\alpha_a\beta_a\right).
\]

Consider a fixed  oriented generator \(c\) (to be specified later). Let

\[
 U'_c=\{u':u'c^{-1}\in U(c)\},
 \qquad
 V'_c=\{v':cv'\in V(c)\}.
\]
Then
\[
 U(c)V(c)=U'_cV'_c.
\]
By the induction hypothesis,
\[
 |U(c)V(c)|
 =
 |U'(c)V'(c)|
 \ge
 d_{k-1}|U'(c)||V'(c)|
 =
 d_{k-1}|U(c)||V(c)|.
\]

Every word in \(U(c)V(c)\) has length at most \(n+k-2\), whereas every
word of \(U(a)V(b)\) for \(a\neq b\), has length exactly \(n+k\). Hence $U_cV_c$
is disjoint from $\bigcup_{a \neq b}U(a)V(b)$ .  

Therefore:
\[
|UV| \geq |U||V|\left(1-\sum_{a \neq c}\alpha_a\beta_a
- (1-d_{k-1})\alpha_c\beta_c\right).
\]
Note that the sets \(U(a)V(a)\) may overlap with one another for different choices
of \(a\), so we only include one of them in the sum.

We now specify \(c\) to be a letter maximizing
\[
        \alpha_c+\beta_c.
\]
Define
\[
        \lambda_a=\frac12(\alpha_a+\beta_a).
\]
Then \((\lambda_a)_{a\in\mathcal A}\) is a probability distribution, \(c\) maximizes
\(\lambda_a\), and
\[
        \alpha_a\beta_a\leq \lambda_a^2
\]
for every \(a\).  Since \(d_{k-1}\geq 2/3\), we have
\[
        1-d_{k-1}\in [0,1/3].
\]
Therefore Proposition~\ref{prop:optimization}, applied with
\(c_0=1-d_{k-1}\), gives
\[
\begin{aligned}
\sum_{a \neq c}\alpha_a\beta_a+(1-d_{k-1})\alpha_c\beta_c
&\leq
\sum_{a \neq c}\lambda_a^2+(1-d_{k-1})\lambda_c^2  \\
&\leq
\frac{1+(1-d_{k-1})}{4}.
\end{aligned}
\]
Consequently
\[
\begin{aligned}
|UV|
&\geq
|U||V|\left(1-\frac{2-d_{k-1}}{4}\right)  \\
&=
|U||V|\left(\frac12+\frac{d_{k-1}}4\right)  \\
&=
d_k|U||V|.
\end{aligned}
\]


Next we prove $\gamma(n,k)\leq d_k$ by constructing an example
of $U \subseteq S_n$ and $V \subseteq S_k$ for which $|UV|=d_k|U||V|$.

Choose two free generators \(a,b\in F_r\). Let
\[
U=\{a^{-1},b^{-1}\}^n
\]
be the set of all negative words of length \(n\) in the letters \(a,b\), and let
\[
V=\{a,b\}^k
\]
be the set of all positive words of length \(k\) in the letters \(a,b\). Then
\(U\subseteq S_n\), \(V\subseteq S_k\), and
\[
|U|=2^n,\qquad |V|=2^k.
\]

It remains to compute \(|UV|\). Write
\[
W_m=\{a,b\}^m
\qquad\text{and}\qquad
W_*=\bigcup_{m\geq 0}W_m.
\]
Then
\[
U=\{x^{-1}:x\in W_n\},
\qquad
V=W_k.
\]
Thus
\[
UV=\{x^{-1}y:x\in W_n,\ y\in W_k\}.
\]

Let \(X\) be the set of pairs \((r,q)\) such that
\[
        r,q\in W_*,
        \qquad
        |q|\leq k,
        \qquad
        |r|-|q|=n-k,
\]
and either \(q\) is the empty word or the first letters of \(r\) and \(q\) are
different.

\begin{claim}
\[
        UV=\{r^{-1}q:(r,q)\in X\}.
\]
\end{claim}

\begin{proof}
First let \(w\in UV\). Then \(w=x^{-1}y\) for some \(x\in W_n\) and
\(y\in W_k\). Let \(p\) be the longest common prefix of \(x\) and \(y\), and write
\[
        x=pr,\qquad y=pq.
\]
Then
\[
        w=(pr)^{-1}pq=r^{-1}q.
\]
Moreover,
\[
        |q|\leq k,
        \qquad
        |r|-|q|=|x|-|y|=n-k.
\]
By maximality of \(p\), either \(q\) is empty or the first letters of \(r\) and \(q\)
are different. Thus \((r,q)\in X\).

Conversely, suppose \((r,q)\in X\). Choose any \(p\in W_{k-|q|}\). Then
\[
        pr\in W_n,\qquad pq\in W_k,
\]
because
\[
        |pr|=k-|q|+|r|=k-|q|+(n-k+|q|)=n
\]
and
\[
        |pq|=k.
\]
Therefore
\[
        r^{-1}q=(pr)^{-1}pq\in UV.
\]
This proves the claim.
\end{proof}

The word \(r^{-1}q\) is reduced for every \((r,q)\in X\), and the reduced word
determines the pair \((r,q)\). Hence different pairs in \(X\) give different elements
of \(UV\), and so
\[
        |UV|=|X|.
\]

For \(0\leq j\leq k\), let
\[
        X_j=\{(r,q)\in X: |q|=j\}.
\]
If \(j=0\), then \(q\) is the empty word and \(|r|=n-k\), so
\[
        |X_0|=2^{n-k}.
\]
If \(1\leq j\leq k\), then \(|q|=j\) and \(|r|=n-k+j\). There are \(2^j\) choices
for \(q\). Once \(q\) is chosen, the first letter of \(r\) must be different from the
first letter of \(q\), and the remaining \(n-k+j-1\) letters of \(r\) are arbitrary.
Thus
\[
        |X_j|=2^j\cdot 2^{n-k+j-1}=2^{n-k+2j-1}.
\]
Therefore
\[
\begin{aligned}
    \frac{|UV|}{|U||V|}
    &=
    2^{-(n+k)}
    \left(2^{n-k}+\sum_{j=1}^{k}2^{n-k+2j-1}\right)  \\
    &=
    4^{-k}+\sum_{j=1}^{k}\frac{2^{2j-1}}{4^k}  \\
    &=
    4^{-k}+\sum_{\ell=0}^{k-1}\frac{1}{2\cdot 4^\ell}  \\
    &=
    \frac{2}{3}+\frac{1}{3\cdot 4^k}
    =
    d_k.
\end{aligned}
\]

This proves \(\gamma(n,k)\leq d_k\).
\end{proof}

\section{Large product growth for hyperbolic groups}
\label{sec:large-product-hyperbolic}

In this section we give a direct proof that hyperbolic groups have linear product
shrinkage.  This is independent of Rapid Decay.  The proof follows the same ``heavy prefix'' strategy as in the free group case, with Gromov products replacing
cancellation length.  The only new point is that in a hyperbolic group a prefix is only determined up to bounded error, and so the marker map is no longer injective;
it has uniformly bounded fibers instead.

The estimate obtained here is sharp in its dependence on the radii, up to the value of the group-dependent constant, whenever \(G\) contains an element of infinite order.

Throughout this section $G$ is a $\delta$-hyperbolic group with respect
to the fixed finite generating set $S$.  
Since the word metric is integer-valued, all vertex Gromov products lie in \(\frac12\mathbb Z\), and hence \(\delta\in \frac12\mathbb Z\).  In particular, \(2\delta\) and \(4\delta\) are integers.

We write
\[
S_m:=\{g\in G: |g|=m\}, \qquad B_m:=\{g\in G: |g|\le m\}.
\]
For each $g\in G$, fix once and for all a geodesic from $1$ to $g$.  If
$0\le t\le |g|$ is an integer, let $g[t]$ denote the point at distance $t$ from
$1$ on this chosen geodesic.  Thus
\[
|g[t]|=t, \qquad |g[t]^{-1}g|=|g|-t.
\]
Since all distances in the Cayley graph are integers, all Gromov products of
vertices lie in \(\frac12\mathbb Z\).  Increasing \(\delta\), if necessary, we assume
throughout this section that \(\delta\in \frac12\mathbb Z\).  Thus \(2\delta\) and
\(4\delta\) are integers.

Let
\[
\varepsilon =
\begin{cases}
0,&\text{if the Cayley graph \(\Gamma(G,S)\) is bipartite},\\
1,&\text{otherwise}.
\end{cases}
\]
Equivalently, \(\varepsilon=0\) precisely when all vertex Gromov products are
integral.
Define
\begin{equation}\label{eq:hyp-product-eta}
\eta:= \eta(\delta) =
\frac{1}{4\,|B_{4\delta+\varepsilon}|\,|B_{4\delta}|}.
\end{equation}
The important point is that \(\eta>0\) depends only on the Cayley graph of \(G\), and
not on the radii of the sets being multiplied.  In the tree case, and in particular for
free groups with the standard generating set, one has \(\delta=0\) and
\(\varepsilon=0\), so this gives \(\eta=1/4\).

\begin{theorem}[Linear product shrinkage for hyperbolic groups]
\label{thm:hyp-linear-product-shrinkage}
Let $G$ be a finitely generated hyperbolic group.  Then $G$ has product shrinkage
at most linear in the radii.  More precisely, with $\eta$ as above, if
$U\subseteq B_n$ and $V\subseteq B_k$ are finite non-empty sets, then
\[
|UV|\ge \frac{\eta}{n+k+1}|U||V|.
\]
In particular, $G$ has large product growth.
\end{theorem}

We prove the theorem after several lemmas.

\begin{lemma}[Nearby prefixes]
\label{lem:hyp-nearby-prefixes}
Let \(x,y\in G\), and let \(s\) be an integer with
\(0\le s\le \min\{|x|,|y|\}\).
\begin{enumerate}[label=\textup{(\arabic*)}]
    \item If
    \[
    (x\mid y)\ge s,
    \]
    then
    \[
    d(x[s],y[s])\le 4\delta.
    \]
    \item If
    \[
    (x\mid y)>s-1,
    \]
    then
    \[
    d(x[s],y[s])\le 4\delta+\varepsilon.
    \]
\end{enumerate}
\end{lemma}

\begin{proof}
Since \(x[s]\) lies on a geodesic from \(1\) to \(x\),
\[
(x[s]\mid x)=s.
\]
Similarly, \((y\mid y[s])=s\).

For (1), Proposition~\ref{prop:easy hyp}(\ref{easy hyp 1}) gives
\[
(x[s]\mid y[s])
\ge
\min\bigl((x[s]\mid x),(x\mid y),(y\mid y[s])\bigr)-2\delta
\ge s-2\delta.
\]
Since \(|x[s]|=|y[s]|=s\), we get
\[
d(x[s],y[s])=2s-2(x[s]\mid y[s])\le 4\delta.
\]

For (2), note that \((x\mid y)>s-1\) implies
\[
        (x\mid y)\ge s-\frac{\varepsilon}{2}.
\]
Indeed, if all Gromov products are integral this is immediate with
\(\varepsilon=0\), and otherwise all Gromov products still lie in
\(\frac12\mathbb Z\).  Applying the same argument as above gives
\[
(x[s]\mid y[s])\ge s-\frac{\varepsilon}{2}-2\delta.
\]
Hence
\[
d(x[s],y[s])\le 4\delta+\varepsilon.
\]
\end{proof}

\begin{lemma}[Moving the basepoint past a prefix]
\label{lem:hyp-move-basepoint}
Let $x,y\in G$, let $p=y[t]$, and suppose
\[
(x\mid y)\le t+A.
\]
Then
\[
(p^{-1}x\mid p^{-1}y)\le A+\delta.
\]
\end{lemma}

\begin{proof}
Let $a=(x\mid y)$.  Since $p=y[t]$, we have $(y\mid p)=t$.  By
$\delta$-hyperbolicity,
\[
(x\mid p)\ge \min\bigl((x\mid y),(y\mid p)\bigr)-\delta
=\min(a,t)-\delta.
\]
Now
\[
(p^{-1}x\mid p^{-1}y)=(x\mid y)_p.
\]
Also $d(p,y)=|y|-t$, and, using Proposition~\ref{prop:easy metric}(\ref{easy metric 1}),
\[
d(p,x)=|x|+t-2(x\mid p).
\]
Therefore
\[
(x\mid y)_p
=\frac12\bigl(d(p,x)+d(p,y)-d(x,y)\bigr)
=(x\mid y)-(x\mid p).
\]
It follows that
\[
(x\mid y)_p\le a-\min(a,t)+\delta
=\max(0,a-t)+\delta.
\]
Since $a\le t+A$, this gives
\[
(p^{-1}x\mid p^{-1}y)=(x\mid y)_p\le A+\delta.
\]
\end{proof}

\begin{lemma}[Bounded fibers under bounded cancellation]
\label{lem:hyp-bounded-cancellation-fibers}
Let \(E\subseteq S_r\), let \(F\subseteq G\) be finite, and fix
\(M_0\ge 0\) such that \(2M_0\) is an integer.  Let
\[
P_{M_0}:=\{(a,b)\in E\times F:(a^{-1}\mid b)\le M_0\}.
\]
Then the multiplication map
\[
P_{M_0}\longrightarrow G, \qquad (a,b)\mapsto ab,
\]
has fibers of size at most \(|B_{2M_0+2\delta}|\).
\end{lemma}

\begin{proof}
Suppose
\[
ab=a'b'=w
\]
with \((a,b),(a',b')\in P_{M_0}\).  Since \(a,a'\in S_r\),
\[
|a|=|a'|=r.
\]
By Proposition~\ref{prop:easy metric}\textup{(\ref{easy metric 4})}, and by
symmetry of the Gromov product,
\[
(a\mid w)=(a\mid ab)=(ab\mid a)
=|a|-(a^{-1}\mid b)\ge r-M_0.
\]
Similarly,
\[
(a'\mid w)\ge r-M_0.
\]
By \(\delta\)-hyperbolicity,
\[
(a\mid a')\ge \min\bigl((a\mid w),(a'\mid w)\bigr)-\delta
\ge r-M_0-\delta.
\]
Therefore
\[
d(a,a')=2r-2(a\mid a')\le 2M_0+2\delta.
\]
So, once \(w\) and one possible value of \(a\) are fixed, every other possible value
of \(a'\) lies in \(aB_{2M_0+2\delta}\).  There are at most
\(|B_{2M_0+2\delta}|\) possibilities for \(a'\), and then \(b'\) is forced by
\[
b'=a'^{-1}w.
\]
\end{proof}

\begin{remark}
\label{rem:hyp-bounded-cancellation-fibers-right}
The same conclusion holds if \(E\subseteq G\) is finite and \(F\subseteq S_s\).
Indeed, apply Lemma~\ref{lem:hyp-bounded-cancellation-fibers} to the inverted
pairs
\[
(b^{-1},a^{-1})\in F^{-1}\times E^{-1}.
\]
The cancellation condition is preserved, since
\[
((b^{-1})^{-1}\mid a^{-1})=(b\mid a^{-1})=(a^{-1}\mid b),
\]
and
\[
b^{-1}a^{-1}=(ab)^{-1}.
\]
Thus fibers of the original multiplication map are in bijection with fibers of
the inverted multiplication map.
\end{remark}

\begin{lemma}[Marker bounded multiplicity]
\label{lem:hyp-marker-bounded}
Let \(j\ge t\ge 0\), let \(p\in S_t\), and let \(W\subseteq S_j\) be such that
\[
v[t]=p \qquad \text{for every } v\in W.
\]
Let \(E\subseteq G\) be finite, and let
\[
P:=\{(u,v)\in E\times W:(u^{-1}\mid v)\le t\}.
\]
Then the multiplication map
\[
P\longrightarrow G,\qquad (u,v)\mapsto uv,
\]
has fibers of size at most \(|B_{4\delta}|\).
\end{lemma}

\begin{proof}
Set
\[
E':=Ep=\{up:u\in E\},
\qquad
W':=p^{-1}W=\{p^{-1}v:v\in W\}.
\]
Since \(p=v[t]\) for every \(v\in W\), we have
\[
W'\subseteq S_{j-t}.
\]
The map
\[
E\times W\longrightarrow E'\times W',
\qquad
(u,v)\mapsto (up,p^{-1}v),
\]
is a bijection, and it preserves products:
\[
(up)(p^{-1}v)=uv.
\]

Let \(P'\) be the image of \(P\) under this bijection.  We claim that
\[
P'\subseteq
\{(a,q)\in E'\times W':(a^{-1}\mid q)\le \delta\}.
\]
Indeed, if \((a,q)=(up,p^{-1}v)\) with \((u,v)\in P\), then
\[
a^{-1}=p^{-1}u^{-1},
\qquad
q=p^{-1}v.
\]
Since \(p=v[t]\) and \((u^{-1}\mid v)\le t\), Lemma~\ref{lem:hyp-move-basepoint}
with \(A=0\) gives
\[
(a^{-1}\mid q)
=(p^{-1}u^{-1}\mid p^{-1}v)
=(u^{-1}\mid v)_p
\le \delta.
\]

Now apply the right-sided form from
Remark~\ref{rem:hyp-bounded-cancellation-fibers-right}, with the second factor
\(W'\subseteq S_{j-t}\) and with \(M_0=\delta\).
Since \(2\delta\) is an
integer, multiplication has fibers of size at most
\[
|B_{2\delta+2\delta}|=|B_{4\delta}|
\]
on
\[
\{(a,q)\in E'\times W':(a^{-1}\mid q)\le \delta\},
\]
and hence also on its subset \(P'\).  Because the bijection above preserves
products, the original multiplication map on \(P\) also has fibers of size at
most \(|B_{4\delta}|\).
\end{proof}

We now introduce the analogue of the prefix and suffix cells used in the free group
argument.  For $E\subseteq S_m$ and $p\in S_t$, define
\[
E[p]:=\{e\in E:e[t]=p\}.
\]
For a finite set $U\subseteq G$, let
\[
U_i:=U\cap S_i,
\qquad
\Lambda^+(U):=\{i\ge 1:U_i\ne\emptyset\}.
\]
For $p\in S_t$, define the inverse-suffix cell
\[
U_i^{-}[p]:=\{u\in U_i:u^{-1}[t]=p\}.
\]
Similarly, for a finite set $V\subseteq G$, write $V_j=V\cap S_j$ and define
\[
V_j^{+}[p]:=\{v\in V_j:v[t]=p\}.
\]
For \(t\ge 1\), we say that \(V_j^{+}[p]\) is heavy if
\[
|V_j^{+}[p]|\ge \frac{|V_j|}{2|B_{4\delta+\varepsilon}|},
\]
and similarly \(U_i^{-}[p]\) is heavy if
\[
|U_i^{-}[p]|\ge \frac{|U_i|}{2|B_{4\delta+\varepsilon}|}.
\]
\begin{lemma}[Heavy-prefix lemma]
\label{lem:hyp-heavy-prefix}
$\delta$-hyperbolic groups satisfy Property $Q(\zeta)$ with $\zeta = \eta(\delta)$. 
That is, let $U,V\subseteq G$ be finite non-empty sets, and suppose that
$\Lambda^+(U)\cup\Lambda^+(V)$ is non-empty.  Then either there exists
$i\in\Lambda^+(U)$ such that
\[
|U_iV|\ge \eta |U_i||V|,
\]
or there exists $j\in\Lambda^+(V)$ such that
\[
|UV_j|\ge \eta |U||V_j|.
\]
\end{lemma}

\begin{proof}
If $V\subseteq S_0=\{1\}$, choose $i\in\Lambda^+(U)$.  Then
\[
|U_iV|=|U_i|\ge \eta |U_i||V|,
\]
since $|V|=1$ and $\eta\le 1$.  The case $U\subseteq S_0$ is symmetric.  We may
therefore assume that both $U$ and $V$ have non-empty positive layers.

Choose a heavy cell of maximal depth among all cells of the forms
$U_i^{-}[p]$ and $V_j^{+}[p]$.  If no positive-depth heavy cell exists, choose any
$j\in\Lambda^+(V)$ and set
\[
p=1, \qquad t=0, \qquad W=V_j.
\]
We consider two cases.

First suppose that either there is no heavy cell, or that the chosen heavy cell lies
in $V$.  Thus
\[
W=V_j^{+}[p]
\]
for some $p\in S_t$, where either $t=0$ and $W=V_j$, or $t\ge 1$ and
$|W|\ge \frac{1}{2|B_{4\delta+\varepsilon}|} |V_j|$.

We prove that
\begin{equation}\label{eq:hyp-UW-large}
|UW|\ge \frac{1}{2|B_{4\delta}|}|U||W|.
\end{equation}
Call a pair \((u,v)\in U\times W\) good if
\[
(u^{-1}\mid v)\le t.
\]
Otherwise call it bad.  Fix \(v\in W\) and fix a layer \(U_i\).  If \(u\in U_i\) is bad,
then
\[
(u^{-1}\mid v)>t.
\]
Since a Gromov product is bounded above by the distance from the basepoint to
each of its two entries, the prefixes at depth \(t+1\) are defined.  By
Lemma~\ref{lem:hyp-nearby-prefixes}(2), applied to \(u^{-1}\) and \(v\) with
\(s=t+1\), we have
\[
d(u^{-1}[t+1],v[t+1])\le 4\delta+\varepsilon.
\]
Thus the bad elements $u\in U_i$ lie in the union of the cells
\[
\{U_i^{-}[r]:r\in S_{t+1}, d(r,v[t+1])\le 4\delta+\varepsilon\}.
\]
There are at most $|B_{4\delta+\varepsilon}|$ such values of $r$.  By maximality of the chosen heavy cell, no cell of depth $t+1$ is heavy.  In the no-heavy case this is true by
assumption.  Hence each of these cells has size strictly less than $\frac{|U_i|}{2|B_{4\delta+\varepsilon}|}$.
The number of bad elements in $U_i$ is therefore less than
\[
|B_{4\delta+\varepsilon}|\cdot
\frac{|U_i|}{2|B_{4\delta+\varepsilon}|}
=\frac12|U_i|.\]
So, for this fixed $v$, at least half of each layer $U_i$ is good.  Summing over
$i$ and then over $v\in W$, the number of good pairs in $U\times W$ is at least
\[
\frac12 |U||W|.
\]
On these good pairs, Lemma~\ref{lem:hyp-marker-bounded} applies, because every
$v\in W$ satisfies $v[t]=p$.  Hence multiplication has fibers of size at most \(|B_{4\delta}|\) on the good pairs, and this proves \eqref{eq:hyp-UW-large}.

If $t=0$, then $W=V_j$, and \eqref{eq:hyp-UW-large} gives
\[
|UV_j|\ge \frac{1}{2|B_{4\delta}|}|U||V_j|
\ge \eta |U||V_j|.
\]
If $t\ge 1$, then $|W|\ge \frac{1}{2|B_{4\delta+\varepsilon}|} |V_j|$, and so
\[
|UV_j|\ge |UW|
\ge \frac{1}{2|B_{4\delta}|}|U||W|
\ge
\frac{1}{4|B_{4\delta+\varepsilon}|\,|B_{4\delta}|}
|U||V_j|
=\eta |U||V_j|.
\]
This gives the second alternative of the lemma.

It remains to handle the case where the chosen heavy cell lies in $U$.  Thus
\[
W=U_i^{-}[p]
\]
for some $p\in S_t$, $t\ge 1$, and
\[
|W|\ge \frac{1}{2|B_{4\delta+\varepsilon}|} |U_i|.
\]
We prove that
\begin{equation}\label{eq:hyp-WV-large}
|WV|\ge \frac{1}{2|B_{4\delta}|}|W||V|.
\end{equation}

Call a pair \((u,v)\in W\times V\) good if \((u^{-1}\mid v)\le t\), and bad
otherwise.  Fix \(u\in W\) and fix a layer \(V_j\).  If \(v\in V_j\) is bad, then
\((u^{-1}\mid v)>t\).  Hence the prefixes at depth \(t+1\) are defined, and
Lemma~\ref{lem:hyp-nearby-prefixes}(2), applied again with \(s=t+1\), gives
\[
d(u^{-1}[t+1],v[t+1])\le 4\delta+\varepsilon.
\]
Thus the bad elements $v\in V_j$ lie in the union of the cells
\[
V_j^{+}[r], \qquad r\in S_{t+1}, \qquad d(r,u^{-1}[t+1])\le 4\delta+\varepsilon.
\]
There are at most \(|B_{4\delta+\varepsilon}|\) such values of \(r\), and by
maximality none of these cells is heavy.  Hence each of these cells has size
strictly less than
\[
        \frac{|V_j|}{2|B_{4\delta+\varepsilon}|}.
\]
Thus fewer than half of the elements of \(V_j\) are bad.
Summing over $j$ and then over $u\in W$, the number of good pairs in $W\times V$ is at
least
\[
\frac12 |W||V|.
\]
We now bound the fibers of multiplication on these good pairs.  Since $u\in W$
means $u^{-1}[t]=p$, every element of $W^{-1}$ has chosen prefix $p$ at depth $t$.
Apply Lemma~\ref{lem:hyp-marker-bounded} to the pair of sets $V^{-1}$ and
$W^{-1}$.  A good pair $(u,v)\in W\times V$ corresponds to
$(v^{-1},u^{-1})\in V^{-1}\times W^{-1}$, and
\[
((v^{-1})^{-1}\mid u^{-1})=(v\mid u^{-1})=(u^{-1}\mid v)\le t.
\]
Therefore the map
\[
(v^{-1},u^{-1})\mapsto v^{-1}u^{-1}
\]
has fibers of size at most \(|B_{4\delta}|\) on the corresponding good pairs.  Taking inverses is a bijection, so the original multiplication map $(u,v)\mapsto uv$ also has fibers of size at most \(|B_{4\delta}|\) on the good pairs.  This proves \eqref{eq:hyp-WV-large}.
Since $|W|\ge \frac{1}{2|B_{4\delta+\varepsilon}|} |U_i|$, we get
\[
|U_iV|\ge |WV|
\ge \frac{1}{2|B_{4\delta}|}|W||V|
\ge
\frac{1}{4|B_{4\delta+\varepsilon}|\,|B_{4\delta}|}
|U_i||V|
=\eta |U_i||V|.
\]
This gives the first alternative and completes the proof.
\end{proof}

As an immediate consequence, we get a radius-independent product estimate for sets
supported on spheres.

\begin{corollary}[Products of spherical layers]
\label{cor:hyp-spherical-product}
If $U\subseteq S_n$ and $V\subseteq S_k$ are finite non-empty sets, then
\[
|UV|\ge \eta |U||V|.
\]
\end{corollary}

\begin{proof}
If $n=0$ or $k=0$, the conclusion is immediate since one of the sets is contained
in $\{1\}$.  If $n,k\ge 1$, then $U=U_n$ and $V=V_k$, so the conclusion follows
from Lemma~\ref{lem:hyp-heavy-prefix}.
\end{proof}

Lemma~\ref{lem:hyp-heavy-prefix} says precisely that \(G\) satisfies property
\(Q(\eta)\).  We therefore obtain the ball--ball estimate from the general
layer-peeling criterion of Section~\ref{subsec:sufficient-linear-shrinkage}.

\begin{proof}[Proof of Theorem~\ref{thm:hyp-linear-product-shrinkage}]
By Lemma~\ref{lem:hyp-heavy-prefix}, the group \(G\) satisfies property \(Q(\eta)\).
The desired estimate follows from Corollary~\ref{cor:Q-linear-product-shrinkage}.
\end{proof}

\begin{remark}
The general layer-peeling theorem gives the stronger layer-sensitive estimate
\[
|UV|\ge \frac{\eta}{|\Lambda^+(U)|+|\Lambda^+(V)|}|U||V|
\]
whenever the denominator is non-zero.  The dependence on the number of non-empty
layers is the direct analogue of the free group ball--ball argument.  Also, when
$G$ contains an element $a$ of infinite order, the linear dependence on the radii is
sharp up to constants: intervals
\[
\{a^{-N},a^{-N+1},\ldots,a^N\}
\]
inside the cyclic subgroup generated by $a$ have product set of size $4N+1$ while
the two factors each have size $2N+1$.
\end{remark}

The preceding corollary gives a radius-independent product estimate
when both factors are supported on fixed spheres.  
Nica \cite{Nica2024} proved a related but different product growth result that
applies when
the left factor is the entire sphere \(S_n\), and the right factor is an arbitrary finite set \(X\subseteq G\), not necessarily supported on one spherical layer.

\begin{theorem}[Expansion of full spheres]\cite[Theorem~2.5]{Nica2024}
\label{thm:hyp-full-sphere-expansion}
Assume that \(G\) is non-elementary.  Then there exists a constant
\(c=c(G,S)>0\) such that, for every \(n\ge 0\) and every finite non-empty
\(X\subseteq G\),
\[
|S_nX|\ge c\,|S_n|\,|X|.
\]
Equivalently, in the notation of \cite[Definition~6.1]{Nica2024},
\[
e(S_n)\ge c|S_n|.
\]
\end{theorem}

Nica's proof uses functional analysis and he asked in \cite[Remark~8.4]{Nica2024} whether this bound can be proved in a ``direct, combinatorial way, without appealing to functional
analytic detours." Here we provide such a proof based on Lemma~\ref{lem:hyp-bounded-cancellation-fibers}.
First we need:


\begin{lemma}[Full spheres have small shadows]
\label{lem:hyp-full-sphere-shadow}
Assume that \(G\) is non-elementary.  Then there exists an integer \(M\ge 0\)
such that, for every \(x\in G\) and every \(n\ge 0\),
\[
\left|\{s\in S_n:(s^{-1}\mid x)>M\}\right|\le \frac12 |S_n|.
\]
\end{lemma}

\begin{proof}
By Coornaert's pure spherical growth theorem \cite{Coornaert1993}, there are
constants \(A_-,A_+>0\) and \(\omega>0\) such that
\[
A_-e^{\omega m}\le |S_m|\le A_+e^{\omega m}
\qquad (m\ge 0).
\]
Let \(Q=|B_{4\delta}|\), and choose an integer \(M\ge 0\) such that
\[
Q\frac{A_+}{A_-}e^{-\omega M}\le \frac12 .
\]
Fix \(x\in G\) and \(n\ge 0\), and put
\[
\mathcal B=\{s\in S_n:(s^{-1}\mid x)>M\}.
\]
If \(\mathcal B=\emptyset\), there is nothing to prove.  Otherwise \(M<n\)
and \(M<|x|\).  For \(s\in\mathcal B\), Lemma~\ref{lem:hyp-nearby-prefixes}
\textup{(1)}, applied to \(s^{-1}\) and \(x\) at depth \(M\), gives
\[
d(s^{-1}[M],x[M])\le 4\delta.
\]
Thus for fixed $x$, there are at most \(Q\) possible values of \(p=s^{-1}[M]\).

For a fixed such \(p\), the map
\[
s\mapsto p^{-1}s^{-1}
\]
injects the set
\[
\{s\in S_n:s^{-1}[M]=p\}
\]
into \(S_{n-M}\), since \(p\) lies on the chosen geodesic from \(1\) to
\(s^{-1}\).  Therefore
\[
|\mathcal B|\le Q|S_{n-M}|
\le Q A_+e^{\omega(n-M)}
\le Q\frac{A_+}{A_-}e^{-\omega M}|S_n|
\le \frac12 |S_n|.
\]
\end{proof}

\begin{proof}[Proof of Theorem ~\ref{thm:hyp-full-sphere-expansion}]
Let \(M\) be as in Lemma~\ref{lem:hyp-full-sphere-shadow}, and set
\[
P=\{(s,x)\in S_n\times X:(s^{-1}\mid x)\le M\}.
\]
For each fixed \(x\in X\), Lemma~\ref{lem:hyp-full-sphere-shadow} gives at
least \(|S_n|/2\) choices of \(s\).  Hence
\[
|P|\ge \frac12 |S_n||X|.
\]
By Lemma ~\ref{lem:hyp-bounded-cancellation-fibers}, applied with
\(E=S_n\), \(F=X\), and \(M_0=M\), multiplication has fibers of size at most
\[
D:=|B_{2M+2\delta}|
\]
on \(P\).  Therefore
\[
|S_nX|\ge \frac{|P|}{D}
\ge \frac{1}{2D}|S_n||X|.
\]
The result follows with \(c=(2D)^{-1}\).
\end{proof}


\section{Product Sets in Groups with Rapid Decay}\label{sec:rd-product-growth}
This final section records a standard analytic source of large product growth.
The relevant product-set estimate is due to Sapir: property RD implies a
one-sided expansion property which he called Rapid Expansion
\cite[Section~4]{sapir2015rapid}.  Since Rapid Expansion immediately implies
large product growth, this gives a broad and robust source of examples.  The
direct geometric arguments of the preceding sections give sharper estimates for
free and hyperbolic groups, but the RD argument applies in much greater
generality.

We begin by fixing the normalization of Rapid Decay that we use.  Let \(G\) be a
finitely generated group, equipped with the word length \(|\cdot|\).  We write
\(\ell^2(G)\) for the Hilbert space of functions \(\xi:G\to\mathbb C\) with
\[
        \|\xi\|_2^2=\sum_{g\in G}|\xi(g)|^2<\infty.
\]
For \(s\in G\), let \(\lambda(s)\) denote the left regular representation on
\(\ell^2(G)\), so that
\[
        (\lambda(s)\xi)(h)=\xi(s^{-1}h).
\]
If \(f:G\to\mathbb C\) is finitely supported, set
\[
        \lambda(f)=\sum_{s\in G}f(s)\lambda(s).
\]

\begin{definition}[Sobolev norms]\label{def:sobolev}
For \(s\ge 0\) and a finitely supported function \(f:G\to\mathbb C\), define
\[
        \|f\|_{\ell^2_s}^2
        =\sum_{g\in G}|f(g)|^2(1+|g|)^{2s}.
\]
Thus \(\|f\|_{\ell^2_0}=\|f\|_2\).
\end{definition}

\begin{definition}[Rapid Decay, {\cite{Jolissaint1990}}]
\label{def:RD}
Let \(G\) be a finitely generated group with word length \(|\cdot|\).  We say that
\(G\) has property \(\mathrm{RD}^s\), with constant \(C>0\), if for every finitely
supported function \(f:G\to\mathbb C\),
\[
        \|\lambda(f)\|\le C\|f\|_{\ell^2_s}.
\]
We say that \(G\) has the Rapid Decay property, or property \(\mathrm{RD}\), if
\(G\) has \(\mathrm{RD}^s\) for some \(s\ge0\) and some \(C>0\).
\end{definition}

The terminology and systematic study of Rapid Decay are due to Jolissaint
\cite{Jolissaint1990}.  The property was first established for free groups by
Haagerup \cite{Haagerup1979}.  It is now known for many classes of groups
appearing in geometric group theory, including hyperbolic groups and groups of
polynomial growth; see, for example, \cite{chatterji2017introRD}.

\begin{definition}[Rapid Expansion, {\cite[Section~4]{sapir2015rapid}}]
\label{def:rapid-expansion}
Let \(G\) be a finitely generated group with word length \(|\cdot|\).  We say
that \(G\) has \emph{Rapid Expansion} if there exists a polynomial \(P\) such
that, for every \(r\ge0\), every finite set \(Y\subseteq B_r\), and every finite
set \(X\subseteq G\),
\[
        |YX|\ge \frac{|Y|\,|X|}{P(r)}.
\]
\end{definition}

Rapid Expansion is stronger than large product growth.  Indeed, if
\(U\subseteq B_n\) and \(V\subseteq B_k\), then applying Rapid Expansion with
\(Y=U\), \(X=V\), and \(r=n\) gives
\[
        |UV|\ge \frac{|U|\,|V|}{P(n)}.
\]
After replacing \(P\) by a non-decreasing polynomial dominating it, this is at
least
\[
        \frac{|U|\,|V|}{P(n+k)}.
\]
Thus Rapid Expansion implies large product growth.  The converse does not
follow formally from the definitions, since large product growth allows the
polynomial loss to depend on the radii of both factors.

Sapir proved that property RD implies Rapid Expansion
\cite[Proposition~4.1]{sapir2015rapid}.  In the normalization of
Definition~\ref{def:RD}, this gives the following explicit large-product-growth
consequence.

\begin{theorem}\label{thm:RD-large-product-growth}
Let \(G\) be a finitely generated group with property \(\mathrm{RD}^s\), with
constant \(C>0\) as in Definition~\ref{def:RD}.  Let \(n,k\ge0\), and let
\[
        U\subseteq B_n,
        \qquad
        V\subseteq B_k
\]
be finite non-empty subsets of \(G\).  Then
\[
        |UV|\ge
        \frac{|U|\,|V|}{C^2(1+\min\{n,k\})^{2s}}.
\]
In particular, \(G\) has large product growth.
\end{theorem}

\begin{proof}
By Sapir's Rapid Expansion estimate \cite[Proposition~4.1]{sapir2015rapid},
translated to the normalization of Definition~\ref{def:RD}, if
\(Y\subseteq B_r\) and \(X\subseteq G\) are finite, then
\[
        |YX|\ge \frac{|Y|\,|X|}{C^2(1+r)^{2s}}.
\]
Applying this with \(Y=U\), \(X=V\), and \(r=n\), we get
\[
        |UV|\ge \frac{|U|\,|V|}{C^2(1+n)^{2s}}.
\]
Applying the same estimate to \(V^{-1}\subseteq B_k\) and
\(U^{-1}\subseteq B_n\), and using
\[
        |UV|=|V^{-1}U^{-1}|,
\]
we also get
\[
        |UV|\ge \frac{|U|\,|V|}{C^2(1+k)^{2s}}.
\]
Taking the better of the two estimates gives
\[
        |UV|\ge
        \frac{|U|\,|V|}{C^2(1+\min\{n,k\})^{2s}}.
\]
Since \((1+\min\{n,k\})^{2s}\le (1+n+k)^{2s}\), this is a polynomial shrinkage
bound in the sense of large product growth.
\end{proof}

\begin{corollary}\label{cor:RD-growth-rate}
Let \(G\) be a finitely generated group with property \(\mathrm{RD}\), and let
\(A\subseteq G\) be a non-empty approximate semigroup.  Then the limit
\[
        \lim_{n\to\infty}\frac{1}{n}\log |A\cap B_n|
\]
exists.
\end{corollary}

\begin{proof}
By Theorem~\ref{thm:RD-large-product-growth}, the group \(G\) has large product
growth.  The conclusion follows from Theorem~\ref{thrm:existence-of-limit}.
\end{proof}

\section*{Acknowledgments}
We thank Bogdan Nica for helpful comments, for pointing us to Sapir's Rapid Expansion property, and for drawing our attention to Remark~8.4 of \cite{Nica2024}.

\bibliographystyle{abbrv} 
\bibliography{bib}
\end{document}